\documentclass[namedate,webpdf,imanum]{ima-authoring-template}

\usepackage{graphicx}
\usepackage{amsmath,amssymb,amsfonts}
\usepackage{amsthm}
\usepackage{booktabs}
\usepackage{mathtools}
\newcommand{\ud}{\mathrm{d}}
\newcommand{\grad}{\mathrm{grad}\,}
\newcommand{\Hess}{\mathrm{Hess}\,}
\newcommand{\D}{\mbox{D}}
\newcommand{\diag}{\mathrm{diag}\,}
\newcommand{\norm}[1]{\left\lVert#1\right\rVert}
\newcommand{\abs}[1]{\left \lvert #1 \right \rvert}
\newcommand{\ip}[1]{\left\langle #1 \right\rangle}
\newcommand{\bx}{\mathbf{x}}
\newcommand{\by}{\mathbf{y}}
\newcommand{\bu}{\mathbf{u}}
\newcommand{\bv}{\mathbf{v}}
\newcommand{\bw}{\mathbf{w}}

\theoremstyle{thmstyleone}
\newtheorem{thm}{Theorem}[section]
\newtheorem{coro}[thm]{Corollary}
\newtheorem{lem}[thm]{Lemma}
\newtheorem{prop}[thm]{Proposition}
\newtheorem{assumption}[thm]{Assumption}
\newtheorem{rem}[thm]{Remark}

\usepackage{etoolbox}
\makeatletter
\patchcmd{\@maketitle}{\itshape\@journaltitle\rm\ (\@copyrightyear) \textbf{\@vol}, \thepage--\thelastpage}{}{}{\PackageError{IMA_paper}{journal line patch failed}{}}
\patchcmd{\@maketitle}{\href{https://doi.org/\@DOI}{https://doi.org/\@DOI}}{}{}{\PackageError{IMA_paper}{DOI patch failed}{}}
\patchcmd{\@maketitle}{[\@history]}{}{}{\PackageError{IMA_paper}{history patch failed}{}}
\def\ps@opening{\let\@oddfoot\@empty\let\@evenfoot\@empty\let\@evenhead\relax\let\@oddhead\relax}
\makeatother

\begin{document}

\DOI{DOI HERE}
\copyrightyear{2026}
\vol{00}
\pubyear{2026}
\access{Advance Access Publication Date: Day Month Year}
\appnotes{Paper}
\copyrightstatement{Published by Oxford University Press on behalf of the Institute of Mathematics and its Applications. All rights reserved.}
\firstpage{1}

\title[Riemannian Hessian analysis of Sobolev gradient flows]{Riemannian Hessian analysis of Sobolev gradient flows for the Gross--Pitaevskii ground state problem}

\author{Shixin Zheng
\address{\orgdiv{Department of Mathematics}, \orgname{University of Maryland}, \orgaddress{College Park, MD 20742, USA}. Email: \href{mailto:siuzheng@umd.edu}{siuzheng@umd.edu}}}
\author{Xiangxiong Zhang*
\address{\orgdiv{Department of Mathematics}, \orgname{Purdue University}, \orgaddress{150 N. University Street, West Lafayette, IN 47907, USA}}}

\authormark{S. Zheng and X. Zhang}

\corresp[*]{Corresponding author: \href{mailto:zhan1966@purdue.edu}{zhan1966@purdue.edu}}

\abstract{Sobolev gradient flows for the Gross--Pitaevskii ground state problem are Riemannian gradient descent methods on the sphere under a metric induced by a discrete elliptic operator. We analyze their convergence through the condition number of the Riemannian Hessian at the ground state. Bounds of the Rayleigh quotient of the Riemannian Hessian give a uniform analysis of these methods for different metrics. The uniform analysis applies to the classical metrics, such as the shifted Laplacian and the energy-adaptive metrics. Inspired by this analysis, we also consider three new metrics. One of them is a Thomas--Fermi metric, whose condition number is bounded uniformly in the interaction strength for the one-dimensional harmonic trap and whose iteration count does not grow in the Thomas--Fermi regime in the numerical tests. The performance suggested by our analysis is consistent with three-dimensional numerical tests with up to one billion unknowns, for both Riemannian gradient descent and Riemannian conjugate gradient methods, whose practical performance is also related to the Hessian condition number. The Thomas--Fermi flow needs the fewest iterations, while the metrics inverted exactly by tensor-product solvers are the fastest in computational time on one GPU card.}

\maketitle

\section{Introduction}

\subsection{The Gross--Pitaevskii ground state}\label{sec:gpe}
The Gross--Pitaevskii ground state model \citep{gross1961structure,pitaevskii1961vortex} of a Bose--Einstein condensate of $\mathcal N$ identical bosons with scattering length $a$ in a trapping potential $V$ is given as a minimization over the whole space,
\[
\mathcal E^{\mathrm{GP}}(\phi) = \int_{\mathbb R^3}\Big(\abs{\nabla \phi}^2 + V\abs{\phi}^2 + 4\pi a\abs{\phi}^4\Big)\ud\bx
\qquad\text{subject to}\qquad \int_{\mathbb R^3}\abs{\phi}^2\ud\bx = \mathcal N,
\]
and the infimum satisfies $E^{\mathrm{GP}}(\mathcal N,a) = \mathcal N\,E^{\mathrm{GP}}(1,\mathcal Na)$ \citep{lieb2001bosons}, so after rescaling to unit mass the interaction enters only through $\mathcal Na$. Numerical computations are usually carried out on a truncated bounded box $\Omega = (-L,L)^d$, which is justified by the exponential decay of the ground state \citep[Theorem~2.5]{bao2013mathematical} and by the convergence of the truncated energy to the whole-space energy as $L\to\infty$ \citep[Lemma~2.2]{lieb2001bosons}. Thus the size of the box is an approximation parameter of the same kind as the mesh size. The rescaled problem is the constrained minimization
\begin{equation}\label{eqn:functional_prob}
\min_{u \in H^1_0(\Omega),\ \int_\Omega \abs{u}^2 \ud\bx = 1} \mathcal E(u), \qquad \mathcal E(u):= \frac{1}{2} \int_\Omega \left(\abs{\nabla u}^2 + V \abs{u}^2\right) \ud \bx + \frac\beta 4 \int_\Omega \abs{u}^4 \ud \bx,
\end{equation}
on $\Omega = (-L,L)^d$, $d = 1,2,3$, with a trapping potential $V \geq 0$ and a coupling constant $\beta > 0$, which is the defocusing case. Its Euler--Lagrange equation is the nonlinear eigenvalue problem: find $\lambda\in\mathbb R$ and $u\in H^1_0(\Omega)$ with $\|u\|_{L^2}=1$ such that
\begin{equation}\label{eqn:fnctional_prob_weak}
(\nabla u, \nabla v) + (Vu, v) + \beta(|u|^2u, v) = \lambda (u,v),\qquad \forall v \in H_0^1(\Omega),
\end{equation}
where $(\cdot,\cdot)$ is the $L^2$ inner product on $\Omega$. For $\beta>0$ the minimizer exists, is unique up to sign and can be chosen strictly positive, and it is the eigenfunction of this eigenvalue problem for the smallest eigenvalue $\lambda$, the chemical potential \citep{cances2010numerical,henning2025gross}. Taking $v = u$ in \eqref{eqn:fnctional_prob_weak} shows that the eigenvalue of a normalized eigenfunction $u$ is determined by $u$ through
\begin{equation}\label{eqn:lambda-continuous}
\lambda(u) := 2\mathcal E(u) + \frac\beta2\int_\Omega\abs{u}^4\,\ud\bx .
\end{equation}
The right-hand side is defined for every $u\in H^1_0(\Omega)$ with $\|u\|_{L^2} = 1$. The coupling constant $\beta = 8\pi a\mathcal N$ \citep[(2.4)]{lieb2001bosons} is proportional to the number of particles, and typical experiments with $a\approx10^{-3}$ and $\mathcal N$ between $10^2$ and $10^7$ correspond to $\beta$ between about $2.5$ and $2.5\times10^5$ \citep[Section~1.3]{bao2013mathematical}. In the regime of large $\beta$ the ground state approaches the Thomas--Fermi profile $\sqrt{(\lambda - V)_+/\beta}$, where $(t)_+ := \max\{t,0\}$, with a boundary layer at the edge of its support that the mesh must resolve.

The fully discretized problem thus has three parameters of interest, the mesh size $h$, the box size $L$ and the coupling constant $\beta$. The limits $h\to0$ and $L\to\infty$ are needed for the discrete ground state to approximate the ground state on $\mathbb R^d$, and large $\beta$ is the physical regime. Computation is more difficult for smaller $h$, larger $L$ and larger $\beta$.

\subsection{Sobolev gradient flows as Riemannian gradient descent}
The Fr\'echet derivative $\mathcal E'(v)\in H^{-1}(\Omega)$, $\ip{\mathcal E'(v), w} = (\nabla v, \nabla w) + (V v, w) + \beta (|v|^2 v, w)$, must be represented in a Hilbert space before it can serve as a descent direction. To do so, we fix an inner product $(\cdot,\cdot)_X$ on $H^1_0(\Omega)$. Classical examples are the $H^1$ inner product and the energy-adaptive inner product \citep{doi:10.1137/18M1230463},
\begin{equation}\label{eqn:two-metrics}
(v,w)_{H^1} := (\nabla v, \nabla w) + (v,w), \qquad
(v,w)_{a_u} := (\nabla v, \nabla w) + (V v, w) + \beta\, (u^2 v, w),
\end{equation}
where $(\cdot,\cdot)_{a_u}$ depends on the current iterate $u$. Every inner product used in this paper can be computed via an elliptic operator $A_X$, $(v,w)_X = (A_Xv,w)$, for instance $A_{a_u} = -\Delta+V+\beta u^2$ for the second inner product in \eqref{eqn:two-metrics}. The Sobolev gradient $\nabla_X\mathcal E(v)$ is the Riesz representative of $\mathcal E'(v)$ with respect to $(\cdot,\cdot)_X$ \citep{neuberger1997sobolev}, defined by $(\nabla_X\mathcal E(v), w)_X = \ip{\mathcal E'(v), w}$ for all $w\in H^1_0(\Omega)$. Thus the Sobolev gradient is one elliptic solve, $\nabla_X\mathcal E(v) = A_X^{-1}\mathcal E'(v)$. For the $a_u$ inner product, $(u,w)_{a_u} = \ip{\mathcal E'(u), w}$ for all $w$ by \eqref{eqn:two-metrics}, so $\nabla_{a_u}\mathcal E(u) = u$ and no solve is needed.

Projecting the Sobolev gradient onto the tangent space $\{w \in H^1_0(\Omega): (u,w) = 0\}$ with respect to $(\cdot,\cdot)_X$ and renormalizing gives the Sobolev gradient descent
\begin{equation}\label{eqn:sgd-continuous}
u^{(k+1)} = \frac{u^{(k)} - \tau_k P^X_{u^{(k)}}\, \nabla_X \mathcal E(u^{(k)})}{\big\| u^{(k)} - \tau_k P^X_{u^{(k)}}\, \nabla_X \mathcal E(u^{(k)}) \big\|_{L^2(\Omega)}},
\end{equation}
with the $X$-orthogonal projection $P^X_u$ onto the tangent space and a step size $\tau_k > 0$. The projection is one more elliptic solve. In fact, $A_X^{-1}u$ spans the $X$-orthogonal complement of the tangent space, since $(A_X^{-1}u, w)_X = (u,w) = 0$ for every $w$ in it, and thus
\begin{equation}\label{eqn:proj-continuous}
P^X_u\,v = v - \frac{(u,v)}{(u,A_X^{-1}u)}\,A_X^{-1}u ,
\qquad
P^X_u\nabla_X\mathcal E(u) = A_X^{-1}\mathcal E'(u) - \frac{(u,A_X^{-1}\mathcal E'(u))}{(u,A_X^{-1}u)}\,A_X^{-1}u .
\end{equation}
For the $a_u$ metric, $\nabla_{a_u}\mathcal E(u) = u$ and the single solve $A_{a_u}^{-1}u$ suffices. For other metrics, one step of \eqref{eqn:sgd-continuous} thus costs two elliptic solves with $A_X$, one for the gradient and one for the projection. Table \ref{tab:metric-list} collects three classical metrics, the modified $H^1$, $a_0$ and $a_u$ metrics, and three new metrics which we consider in this paper. The new metrics are variants of the classical ones and are inspired by the Riemannian Hessian analysis in Section \ref{sec:rayleigh}.

\begin{table}[htbp]
\centering\footnotesize
\setlength{\tabcolsep}{6pt}
\begin{tabular}{llc>{\raggedright\arraybackslash}p{0.3\textwidth}}
\hline
classical metric & elliptic operator $A_X$ & solves/iter. & discussed in\\
\hline
modified $H^1$ & $-\Delta + \alpha$ & 2 & \citet{chen2024fully}\\
$a_0$ & $-\Delta + V$ & 2 & \citet{danaila2010new,danaila2017computation,chen2024convergence}\\
$a_u$ & $-\Delta + V + \beta u^2$ & 1 & \citet{doi:10.1137/18M1230463,hauck2024positivity}\\
\hline
new metric & elliptic operator $A_X$ & solves/iter. & discussed in\\
\hline
modified $a_0$ & $-\Delta + V_1$ & 2 & Section \ref{sec:modified-a0}\\
modified $a_u$ & $-\Delta + V + 3\beta u^2$ & 2 & Section \ref{sec:modified-au}\\
Thomas--Fermi & $-\Delta + 2\beta u^2 + (V-\lambda(u))_+ + \ell\lambda(u)^{1/3}$ & 2 & Section \ref{sec:tf-metric}\\
\hline
\end{tabular}
\caption{Here $\alpha\geq0$ is a constant, $V_1(\bx) = \sum_{j=1}^d v_j(x_j)$ is a separable part of the potential, $\lambda(u)$ is defined by \eqref{eqn:lambda-continuous} for the current iterate $u$ and $\ell > 0$ is a fixed parameter.}
\label{tab:metric-list}
\end{table}

The iteration \eqref{eqn:sgd-continuous} is a Riemannian gradient descent on the sphere, in which the projected Sobolev gradient is the Riemannian gradient for the metric induced by $(\cdot,\cdot)_X$ \citep{danaila2017computation} and the normalization is a retraction \citep[Sect.~4.1]{absil2008optimization}. It can also be regarded as a preconditioned projected gradient descent with the preconditioner $A_X^{-1}$ \citep{antoine2017efficient,henning2025gross}. For the $a_u$ metric, the step with $\tau_k = 1$ is an inverse iteration for the linearized operator, and $\tau_k \in (0,1)$ gives the damped inverse iteration \citep{doi:10.1137/18M1230463,henning2025gross}. After discretization in space, with quadrature for the nonlinear terms, the metrics become matrices, and Section \ref{sec:scheme} formulates the iteration as a Riemannian gradient descent on a sphere in $\mathbb R^N$.

\subsection{Main results and related work}
In this paper we analyze the Riemannian Hessian of the discrete energy at the ground state for an arbitrary metric induced by a discrete elliptic operator, and we compare metrics through its conditioning, as in our earlier work on fixed-rank matrix optimization \citep{zheng2025threemetrics}. At a critical point this Hessian is the second variation of the discrete energy along the normalization retraction (Lemma \ref{lem:second-variation}). It is thus the same quadratic form for every metric, and its Rayleigh quotient is the ratio of a metric-independent constrained Hessian to the metric operator. This gives two-sided bounds at the ground state (Theorem \ref{thm:rayleigh-quotient-bounds}), closed-form bounds on the condition number $\kappa$ (Corollary \ref{coro:conditioning}) and local rates, in particular the rate $(\kappa-1)/(\kappa+1)$ of the fixed-step iteration at the optimal step (Theorem \ref{thm:asymptotic-rate} and Proposition \ref{thm:local-convergence}). For the $a_u$ metric, the lower bound on the spectrum is asymptotically exact in the Thomas--Fermi regime for the one-dimensional harmonic trap, where the condition number grows like the two-thirds power of the ground state eigenvalue (Corollary \ref{thm:TF}). The corresponding bound based on the spectral gap alone is not sharp \citep[Remark~3.3]{doi:10.1142/S0218202523500343}, and Appendix \ref{app:spectra} compares both bounds with exact spectra in one, two and three dimensions. For the normalized $a_0$ flow, an explicit bound for the local contraction rate is listed as open in \citet[Section~5.2.2]{henning2025gross}, and we give one for the fully discretized flow (Corollary \ref{coro:conditioning} and Theorem \ref{thm:asymptotic-rate}). The bound for the modified $H^1$ metric grows with the box size $L$, as do its iteration counts in the tests. We also prove convergence of the whole iterate sequence to a critical point from any initial value (Theorem \ref{thm:global}).

A metric whose operator is closer to the constrained Hessian is better conditioned. Inspired by this analysis, we also consider three new metrics. The modified $a_0$ metric keeps only the separable part $V_1$ of the potential, and its operator $-\Delta+V_1$ is inverted exactly by tensor-product solvers. It is motivated by \citet{liu2026gpu}, where $(-\Delta+V_1)^{-1}$ preconditions the inner solves of the $a_u$ flow and gives mesh- and domain-independent conditioning when $V_1$ is confining. The Riemannian Hessian analysis shows that the metric inherits the robustness in $L$. The modified $a_u$ metric asymptotically halves the condition number of the $a_u$ metric (Proposition \ref{prop:modified-au}). The Thomas--Fermi metric is a Thomas--Fermi approximation of the constrained Hessian. For the one-dimensional harmonic trap, it is not uniformly equivalent to the $a_u$ metric, and its condition number is bounded uniformly in $\beta$ (Proposition \ref{prop:tf-metric} and Remark \ref{rem:tf-orders}). In three-dimensional tests with up to $10^9$ unknowns (Section \ref{sec:numerics}), the Thomas--Fermi flow needs $8$ iterations against $35$ for the $a_u$ flow, and its iteration count does not grow over a factor of $64$ in $\beta$. Nevertheless, the modified $H^1$ and modified $a_0$ metrics are the fastest in wall clock time, because their operators are separable. The modified $a_0$ metric remains the fastest as $L$ grows.

The literature on computing Gross--Pitaevskii ground states is extensive, and we refer to the review \citep{henning2025gross}. The gradient flow with discrete normalization \citep{bao2004computing}, a semi-implicit discretization of the projected $L^2$ gradient flow, remains the most widely used method. The classical metrics of Table \ref{tab:metric-list} are those of the $H^1$ flow \citep{kazemi2010minimizing,chen2024fully}, of the $a_0$ flow and its Riemannian gradient and conjugate gradient versions \citep{danaila2010new,danaila2017computation}, and of the $a_u$ flow \citep{doi:10.1137/18M1230463}, which was extended to Kohn--Sham energies in \citet{altmann2022energy}. Other approaches include preconditioned gradient methods in the $L^2$ metric \citep{antoine2017efficient}, second-order flows \citep{chen2023second}, the $J$-method \citep{altmann2021j} and Riemannian Newton methods \citep{altmann2024riemannian}, and their convergence theory has been developed mostly at the continuous level \citep{kazemi2010minimizing,faou2018convergence,doi:10.1137/18M1230463,zhang2022exponential,chen2024convergence,henning2024rotating,zhang2025rotating}. The basic inverse iteration is the $a_u$ flow with the step $\tau_k = 1$. It converges locally at the rate $\lambda_1/\lambda_2$, the ratio of the first two eigenvalues of the linearized operator \citep[Theorem~3.2]{doi:10.1142/S0218202523500343}. Its sharp rate is the largest modulus $|\mu_1|$ of the eigenvalues of a weighted eigenvalue problem, and $|\mu_1|<\lambda_1/\lambda_2$ for every $\beta>0$ \citep[Remark~3.3]{doi:10.1142/S0218202523500343}. Up to a shift, this weighted eigenvalue problem is the eigenvalue problem of the Riemannian Hessian under the $a_u$ metric (Section \ref{sec:local}). For the damped inverse iteration with a fixed damping parameter $\tau\in(0,1]$ the rate is the convex combination $|1-\tau|+\tau|\mu_1|\leq|1-\tau|+\tau\lambda_1/\lambda_2$ \citep[Theorem~6.2]{doi:10.1142/S0218202523500343}, \citep[Theorem~5.17]{henning2025gross}, and Theorem \ref{thm:asymptotic-rate} gives it as $\max_j|1-\tau+\tau\mu_j|$, the sharper value noted in \citet[Section~6.1]{doi:10.1142/S0218202523500343} (Section \ref{sec:local}). For the gradient flow with discrete normalization, energy dissipation and global convergence were established recently \citep{fengtangwang2025discrete}, and a unified framework of preconditioned Riemannian gradient methods with a local rate derived from a Polyak--{\L}ojasiewicz inequality was proposed in \citet{fengtang2025preconditioned}. Riemannian metrics based on regularized Lagrangian Hessians have also been studied \citep{altmann2026riemannian,fengtang2025preconditioned}. Results for the fully discretized problem are fewer: a priori error estimates \citep{cances2010numerical,henning2014two,henning2023optimal,henning2025discrete}, global convergence and a mesh-independent local rate of the discrete modified $H^1$ flow \citep{chen2024fully}, and positivity and uniqueness of the discrete ground state with global convergence of the discrete $a_u$ flow from nonnegative data \citep{hauck2024positivity}. For monotone discrete Laplacians, our recent work \citep{li2026splitting} proves global convergence to the positive discrete ground state from every positive normalized initial vector for a Davis--Yin operator splitting scheme, which treats the potential and the interaction explicitly and inverts only a shifted Laplacian in each iteration. The Thomas--Fermi asymptotics used in Corollary \ref{thm:TF} were established in \citet{gallo2011thomas} for radial harmonic traps, and the boundary layer analysis was extended to general potentials in two dimensions in \citet{karali2015ground}.

\subsection{Organization of the paper}
Section \ref{sec:scheme} sets up the discrete problem, the six metrics and the iteration. Section \ref{sec:rayleigh} contains the analysis of the Hessian at a critical point, the two-sided bounds, the three new metrics and the Thomas--Fermi asymptotics of the $a_u$ metric. Section \ref{sec:convergence} turns the bounds into local rates and records global convergence, Section \ref{sec:numerics} reports the three-dimensional experiments and Section \ref{sec:conclusions} concludes. Appendix \ref{app:spectra} tests the bounds against exact spectra of the Riemannian Hessian, and Appendix \ref{app:riemannian-connection} contains the Riemannian connection, the Hessian at an arbitrary point and the vector transport.

\section{Riemannian gradient descent}\label{sec:scheme}
\subsection{Discretization}\label{sec:fem}
We discretize \eqref{eqn:functional_prob} by the $Q^k$ finite element method with tensor-product Gauss--Lobatto quadrature on a rectangular mesh, following \citet[Section~2]{chen2024fully}, to which we refer for the details and for the $P^k$ element on a simplicial mesh. Let $\{\bx_i\}_{i=1}^N$ be the interior quadrature nodes, which are also the interpolation nodes of the Lagrange basis $\{\phi_i\}_{i=1}^N$ of the finite element space $V_h^0\subset H^1_0(\Omega)$, and let $w_i>0$ be the quadrature weights. Every $u_h\in V_h^0$ is identified with its vector of nodal values $\bu = [u_h(\bx_1),\ldots,u_h(\bx_N)]^T\in\mathbb R^N$, and the discrete $L^2$ inner product is
\[
\ip{\bu,\bv}_h := \bu^T\mathbb M\bv = \sum_{i=1}^N w_i u_i v_i, \qquad \norm{\bu}_h := \ip{\bu,\bu}_h^{1/2}, \qquad \mathbb M := \diag(w_1,\ldots,w_N).
\]
With the potential matrix $\mathbb V := \diag(V(\bx_1),\ldots,V(\bx_N))$ and the stiffness matrix $\mathbb S_{ij} := \ip{\nabla\phi_i,\nabla\phi_j}_h$, the energy \eqref{eqn:functional_prob} with all integrals replaced by quadrature is the fully discrete energy on $\mathbb R^N$,
\begin{equation}\label{eqn:discrete-energy}
E(\bu) = \frac12\bu^T\mathbb S\bu + \frac12\bu^T\mathbb M\mathbb V\bu + \frac\beta4(\bu^2)^T\mathbb M\bu^2,
\end{equation}
to be minimized over the sphere $\mathcal M := \{\bu\in\mathbb R^N: \ip{\bu,\bu}_h = 1\}$, where powers of vectors such as $\bu^2$ are taken componentwise. Its Euler--Lagrange equation is the discrete eigenproblem
\begin{equation}\label{eqn:matrix-form-1}
A_\beta(\bu)\bu = \lambda\bu, \qquad A_\beta(\bu) := -\Delta_h + \mathbb V + \beta\diag(\bu^2), \qquad \ip{\bu,\bu}_h = 1,
\end{equation}
that is, $\mathbb S\bu + \mathbb M\mathbb V\bu + \beta\mathbb M\bu^3 = \lambda\mathbb M\bu$, where $-\Delta_h := \mathbb M^{-1}\mathbb S$. The matrix $-\Delta_h$ is not symmetric in general. However, every operator $-\Delta_h + \mathbb D$ with a diagonal $\mathbb D$, including $-\Delta_h$ itself, is self-adjoint with respect to $\ip{\cdot,\cdot}_h$, since $\mathbb M(-\Delta_h+\mathbb D) = \mathbb S + \mathbb M\mathbb D$ is symmetric. Such an operator $A$ has real eigenvalues $\lambda_1(A)\leq\cdots\leq\lambda_N(A)$, which are the eigenvalues of the symmetric pencil $(\mathbb MA,\mathbb M)$, and the Courant--Fischer min-max principle holds for them in $\ip{\cdot,\cdot}_h$. We also write $\mathbb I$ for the identity matrix and $A_{3\beta}(\bu) := -\Delta_h + \mathbb V + 3\beta\diag(\bu^2)$ for the derivative of $\bu\mapsto A_\beta(\bu)\bu$. For convenience, we use $A_\beta$ and $A_{3\beta}$ to denote $A_\beta(\bu^*)$ and $A_{3\beta}(\bu^*)$, where $\bu^*$ is the ground state.

On a uniform mesh, the $Q^1$ scheme with this quadrature reduces to the second-order finite difference scheme in any dimension, and so does the $P^1$ scheme in one dimension. Part of the analysis uses the sign structure of the scheme.
\begin{assumption}[Monotone scheme]\label{asp:monotone}
The stiffness matrix $\mathbb S$ is an irreducible Stieltjes matrix, that is, symmetric positive definite with nonpositive off-diagonal entries and an irreducible sparsity pattern, and its row sums are nonnegative.
\end{assumption}
Assumption \ref{asp:monotone} holds for the second-order finite difference scheme in any dimension and for the $P^1$ finite element method with mass lumping on a simplicial mesh satisfying an angle condition, which in two dimensions means a Delaunay triangulation \citep{chen2024fully}. The Perron--Frobenius theorem and the convexity of the energy in $\bu^2$ give the following lemma \citep[Section~3]{chen2024fully}.
\begin{lem}[Simple smallest eigenvalue and positive ground state]\label{lem:perron}
Under Assumption \ref{asp:monotone}, for every $\bu\in\mathbb R^N$ the matrices $\mathbb MA_\beta(\bu) = \mathbb S + \mathbb M(\mathbb V + \beta\diag(\bu^2))$ and $\mathbb MA_{3\beta}(\bu)$ are irreducible Stieltjes matrices. Thus $A_\beta(\bu)$ and $A_{3\beta}(\bu)$ are nonsingular M-matrices with entrywise positive inverses, and their smallest eigenvalues are simple with entrywise positive eigenvectors. In particular, the discrete ground state $\bu^*$ is unique up to sign and strictly positive, and it is an eigenvector of the simple smallest eigenvalue $\lambda_1(A_\beta)$ of $A_\beta$.
\end{lem}

\subsection{Riemannian metrics}\label{sec:classical-metrics}
A Riemannian metric on $\mathbb R^N$ is specified at each point $\bu\in\mathbb R^N$ by a matrix $B(\bu)$, self-adjoint and positive definite with respect to $\ip{\cdot,\cdot}_h$, through
\begin{equation}\label{eqn:metric}
g_\bu(\xi,\eta) := \ip{B(\bu)\xi,\eta}_h = \xi^T\mathbb M B(\bu)\eta, \qquad \xi,\eta\in\mathbb R^N .
\end{equation}
We call $B(\bu)$ the metric operator. The metric is called \emph{Euclidean}, or constant, if $B$ does not depend on $\bu$, and \emph{$\bu$-dependent} otherwise. The discrete counterparts of the six metrics of Table \ref{tab:metric-list}, with $-\Delta_h$ in place of $-\Delta$, are
\begin{equation}\label{eqn:six-metrics}
\begin{aligned}
&\text{modified } H^1\text{:}\ \ B = -\Delta_h + \alpha\mathbb I,\ \alpha\geq0, &&\qquad \text{modified } a_0\text{:}\ \ B = A_1 := -\Delta_h + \mathbb V_1, \\
&a_0\text{:}\ \ B = -\Delta_h + \mathbb V, &&\qquad \text{modified } a_u\text{:}\ \ B(\bu) = A_{3\beta}(\bu), \\
&a_u\text{:}\ \ B(\bu) = A_\beta(\bu), &&\qquad \text{Thomas--Fermi:}\ \ B(\bu) = B_{\mathrm{TF}}(\bu),
\end{aligned}
\end{equation}
where
\begin{equation}\label{eqn:tf-metric}
B_{\mathrm{TF}}(\bu) := -\Delta_h + 2\beta\diag(\bu^2) + \diag\big((\mathbb V - \lambda(\bu))_+\big) + \ell\lambda(\bu)^{1/3}\,\mathbb I
\end{equation}
with the positive part $(\cdot)_+$ taken entrywise, a parameter $\ell>0$ and the discrete counterpart of \eqref{eqn:lambda-continuous},
\begin{equation}\label{eqn:lambda-u}
\lambda(\bu) := \ip{A_\beta(\bu)\bu,\bu}_h = 2E(\bu) + \tfrac\beta2\ip{\bu^2,\bu^2}_h ,
\end{equation}
which is nonnegative for every $\bu\in\mathbb R^N$. For the modified $a_0$ metric, we split $\mathbb V = \mathbb V_1 + \mathbb V_2$, where $\mathbb V_1\geq0$ samples a \emph{separable} potential $V_1(\bx) = \sum_{j=1}^d v_j(x_j)$ and $\mathbb V_2$ is bounded uniformly in $h$. Each of the six operators has the form $B(\bu) = -\Delta_h+\mathbb D$ with a nonnegative diagonal matrix $\mathbb D$. Thus $\mathbb M B(\bu) = \mathbb S + \mathbb M\mathbb D$ is symmetric positive definite, and \eqref{eqn:metric} is a metric on all of $\mathbb R^N$. Applying $B(\bu)^{-1}$ requires one sparse linear solve with $\mathbb S + \mathbb M\mathbb D$. On rectangular meshes with $Q^k$ elements and a separable diagonal $\mathbb D$, as in the modified $H^1$ and modified $a_0$ metrics, this solve is carried out exactly by tensor-product direct solvers at $O(N^{1+1/d})$ cost \citep{liu2024simple,liu2026gpu}, while the other four operators need an inner iterative solver.

For a map $F$ on $\mathbb R^N$, we write $\D F(\bx)[\xi] := \frac{d}{dt}F(\bx+t\xi)\big|_{t=0}$ for its directional derivative. The gradient of $E$ in $(\mathbb R^N,g)$ \citep[(3.31)]{absil2008optimization} is the Riesz representative of $\D E(\bu)[\eta] = \ip{A_\beta(\bu)\bu,\eta}_h$ in $g_\bu$, which is the discrete Sobolev gradient
\begin{equation}\label{eqn:gradient-ambident}
\nabla_X E(\bu) = B(\bu)^{-1}A_\beta(\bu)\bu .
\end{equation}
For the $a_u$ metric, $\nabla_XE(\bu) = \bu$.

\subsection{The manifold and the iteration}\label{sec:submanifold}
The sphere $\mathcal M$ is an embedded submanifold of $(\mathbb R^N,g)$ by the submersion theorem \citep[Prop.~3.3.3]{absil2008optimization}, with tangent and normal spaces
\begin{equation}\label{eqn:tangent-normal-space}
T_\bu\mathcal M = \{\xi\in\mathbb R^N: \ip{\xi,\bu}_h = 0\}, \qquad N_\bu\mathcal M = \{s\,B(\bu)^{-1}\bu: s\in\mathbb R\}.
\end{equation}
The tangent space is the kernel of the derivative $\xi\mapsto2\ip{\xi,\bu}_h$ of the constraint $\ip{\bu,\bu}_h = 1$ \citep[(3.19)]{absil2008optimization}. The normal space is the $g_\bu$-orthogonal complement of the tangent space \citep[Sect.~3.6.1]{absil2008optimization}. It is spanned by $B(\bu)^{-1}\bu$, since $g_\bu(B(\bu)^{-1}\bu,\xi) = \ip{\bu,\xi}_h = 0$ for every $\xi\in T_\bu\mathcal M$. The $g_\bu$-orthogonal projection onto $T_\bu\mathcal M$ and the Riemannian gradient of $E$ on $\mathcal M$ are thus
\begin{equation}\label{eqn:projection}
\begin{gathered}
P_\bu\bw = \bw - \frac{\ip{\bu,\bw}_h}{\ip{\bu,B(\bu)^{-1}\bu}_h}\,B(\bu)^{-1}\bu, \\
\grad E(\bu) = P_\bu\nabla_XE(\bu) = B(\bu)^{-1}\big(A_\beta(\bu)\bu - \lambda_B(\bu)\bu\big),
\end{gathered}
\end{equation}
where $\lambda_B(\bu) := \ip{\bu,B(\bu)^{-1}A_\beta(\bu)\bu}_h/\ip{\bu,B(\bu)^{-1}\bu}_h$ is the metric-weighted Rayleigh quotient. The Riemannian gradient is the projection of the ambient gradient \eqref{eqn:gradient-ambident} onto the tangent space \citep[(3.37)]{absil2008optimization}, and both formulas are the matrix form of \eqref{eqn:proj-continuous}. Thus $\grad E(\bu) = 0$ exactly when $\bu$ is an eigenvector of $A_\beta(\bu)$, that is, a solution of \eqref{eqn:matrix-form-1}, and pairing with $\bu$ identifies the eigenvalue as $\lambda(\bu)$ of \eqref{eqn:lambda-u}, where also $\lambda_B(\bu) = \lambda(\bu)$. The Riemannian gradient descent for minimizing $E$ over $\mathcal M$ with step size $\tau_k>0$ is
\begin{equation}\label{eqn:iteration}
\bu_{k+1} = R_{\bu_k}\big(-\tau_k\grad E(\bu_k)\big), \qquad R_\bu(\xi) := \frac{\bu+\xi}{\norm{\bu+\xi}_h} = \frac{\bu+\xi}{\sqrt{1+\norm{\xi}_h^2}}, \quad \xi\in T_\bu\mathcal M,
\end{equation}
where the normalization $R_\bu$ is a retraction \citep[Def.~4.1.1 and Example~4.1.1]{absil2008optimization}, that is, $R_\bu(0) = \bu$ and $\frac{d}{dt}R_\bu(t\xi)\big|_{t=0} = \xi$ for every $\xi\in T_\bu\mathcal M$. The retraction replaces the exponential map of $(\mathcal M,g)$ \citep[Sect.~5.4]{absil2008optimization}, which has no closed form for the $\bu$-dependent metrics. This iteration is the matrix form of \eqref{eqn:sgd-continuous}. The retraction $R_\bu$ maps $T_\bu\mathcal M$ bijectively onto the open hemisphere $\{\bw\in\mathcal M: \ip{\bw,\bu}_h>0\}$, with inverse $\bw\mapsto\bw/\ip{\bu,\bw}_h - \bu$. In fact, $\ip{R_\bu(\xi),\bu}_h = (1+\norm{\xi}_h^2)^{-1/2} > 0$, and pairing $R_\bu(\xi) = \bw$ with $\bu$ gives $\bu+\xi = \bw/\ip{\bu,\bw}_h$.

\section{Riemannian Hessian analysis}\label{sec:rayleigh}
At a critical point, the Rayleigh quotient of the Riemannian Hessian in a tangent direction $\eta$ is the second variation of $E$ in this direction divided by $g_\bu(\eta,\eta)$. Only the denominator depends on the metric (Section \ref{sec:hessian-critical}). Comparing the numerator with the denominator gives two-sided bounds at the ground state for all metrics at once (Section \ref{sec:bounds}), following the approach in \citet{zheng2025threemetrics}. Inspired by this comparison, we also consider three new metrics (Sections \ref{sec:modified-a0} to \ref{sec:tf-metric}). Section \ref{sec:tf-metric} also derives the sharp Thomas--Fermi asymptotics of the $a_u$ metric from \citet{gallo2011thomas}. Appendix \ref{app:spectra} tests the bounds against exact spectra of the Riemannian Hessian.

\subsection{The Hessian at a critical point}\label{sec:hessian-critical}
The Riemannian Hessian of $E$ on $\mathcal M$ at $\bu$ is the linear map $\Hess E(\bu)$ of $T_\bu\mathcal M$ into itself defined by $\Hess E(\bu)[\xi] = \nabla_{\xi}\grad E$, where $\nabla$ is the Riemannian connection of $(\mathcal M,g)$ \citep[Def.~5.5.1]{absil2008optimization}, and its Rayleigh quotient is
\begin{equation}\label{eqn:rq-def}
\rho_\bu(\eta) := \frac{g_\bu(\Hess E(\bu)[\eta],\eta)}{g_\bu(\eta,\eta)}, \qquad \eta\in T_\bu\mathcal M\setminus\{0\}.
\end{equation}
The Hessian is self-adjoint with respect to $g_\bu$ \citep[Prop.~5.5.3]{absil2008optimization}, so its eigenvalues are real and its extreme eigenvalues are the extreme values of $\rho_\bu$.

\begin{lem}[Second variation at a critical point]\label{lem:second-variation}
Let $\bu^*\in\mathcal M$ be a critical point of $E$ on $\mathcal M$, so that $A_\beta(\bu^*)\bu^* = \lambda(\bu^*)\bu^*$, and let $\eta\in T_{\bu^*}\mathcal M$. Then, for every metric \eqref{eqn:metric} that is smooth near $\bu^*$ and every smooth curve $\bu(t)$ on $\mathcal M$ with $\bu(0) = \bu^*$ and $\bu'(0) = \eta$,
\begin{equation}\label{eqn:second-variation}
g_{\bu^*}(\Hess E(\bu^*)[\eta],\eta) = \frac{d^2}{dt^2}E(\bu(t))\Big|_{t=0} = \ip{\big(A_{3\beta}(\bu^*) - \lambda(\bu^*)\mathbb I\big)\eta,\eta}_h .
\end{equation}
\end{lem}
\begin{proof}
The chain rule and the compatibility of $\nabla$ with the metric \citep[Theorem~5.3.1]{absil2008optimization} give $\frac{d^2}{dt^2}E(\bu(t)) = g_{\bu(t)}(\Hess E(\bu(t))[\bu'(t)],\bu'(t)) + g_{\bu(t)}(\grad E(\bu(t)),\frac{\D}{dt}\bu'(t))$, where $\frac{\D}{dt}$ is the covariant derivative along the curve \citep[Sect.~5.4]{absil2008optimization}. Along the curve, these two properties are the chain rule and the product rule of \citet[Theorem~5.29]{boumal2023intromanifolds}. The second term vanishes at $t=0$ because $\grad E(\bu^*) = 0$. This is the first equality. For curves $t\mapsto R_{\bu^*}(t\eta)$ generated by a retraction, it is also a consequence of \citet[Prop.~5.5.6]{absil2008optimization}. For the second equality, $E$ is a polynomial on $\mathbb R^N$, so $\frac{d^2}{dt^2}E(\bu(t))|_{t=0} = \D^2E(\bu^*)[\eta,\eta] + \D E(\bu^*)[\bu''(0)]$ with $\D^2E(\bu^*)[\eta,\eta] = \ip{A_{3\beta}(\bu^*)\eta,\eta}_h$ and $\D E(\bu^*)[\bw] = \ip{A_\beta(\bu^*)\bu^*,\bw}_h = \lambda(\bu^*)\ip{\bu^*,\bw}_h$ by the eigenvalue equation, while differentiating $\ip{\bu(t),\bu(t)}_h = 1$ twice at $t=0$ gives $\ip{\bu^*,\bu''(0)}_h = -\ip{\eta,\eta}_h$.
\end{proof}

The right-hand side of \eqref{eqn:second-variation} depends neither on the metric nor on the retraction. Only the denominator of \eqref{eqn:rq-def} depends on the metric, and at the ground state $\bu^*$, where $\lambda(\bu^*) = \lambda_1(A_\beta)$ under the assumption of Section \ref{sec:bounds},
\begin{equation}\label{eqn:rho-hweighted}
\rho_{\bu^*}(\eta) = \frac{\ip{H^*\eta,\eta}_h}{\ip{B^*\eta,\eta}_h}, \qquad H^* := A_{3\beta}(\bu^*) - \lambda_1(A_\beta(\bu^*))\mathbb I, \qquad B^* := B(\bu^*),
\end{equation}
where the constrained Hessian $H^*$ is independent of the metric. Bounding the Rayleigh quotient thus amounts to comparing the metric operator $B^*$ with $H^*$. Appendix \ref{app:riemannian-connection} gives the Hessian at arbitrary points, which requires the Riemannian connection, and its Lemma \ref{lem:riemannian-hess} reduces to \eqref{eqn:second-variation} at a critical point.

Lemma \ref{lem:second-variation} requires a metric that is smooth near $\bu^*$. This holds for the first five metrics of \eqref{eqn:six-metrics}, and for the Thomas--Fermi metric exactly when $\mathbb V_i\neq\lambda_1(A_\beta)$ for every node $i$ (Appendix \ref{app:riemannian-connection}). If $\mathbb V_i = \lambda_1(A_\beta)$ at some node, then $B_{\mathrm{TF}}$ is Lipschitz but not differentiable at $\bu^*$, and we define $\Hess E(\bu^*)$ as the operator on $T_{\bu^*}\mathcal M$ that is self-adjoint with respect to $g_{\bu^*}$ and satisfies $g_{\bu^*}(\Hess E(\bu^*)[\eta],\xi) = \ip{H^*\eta,\xi}_h$ for all $\eta,\xi\in T_{\bu^*}\mathcal M$. Whenever the metric is smooth near $\bu^*$, this operator is the Riemannian Hessian by Lemma \ref{lem:second-variation} and the polarization identity \citep[(5.36)]{absil2008optimization}. Thus \eqref{eqn:second-variation} and \eqref{eqn:rho-hweighted} hold for every metric, and all results of this section use only \eqref{eqn:rho-hweighted}.

\subsection{Two-sided bounds at the ground state}\label{sec:bounds}
Throughout this section $\bu^*$ is the ground state. We assume that $\bu^*$ is an eigenvector of the smallest eigenvalue $\lambda_1(A_\beta)$ of $A_\beta$ and that this eigenvalue is simple, so that $\lambda(\bu^*) = \lambda_1(A_\beta)$. Both properties hold under Assumption \ref{asp:monotone} by Lemma \ref{lem:perron}. We write
\begin{equation}\label{eqn:defect}
Z := H^* - B^*, \qquad \zeta_+ := \max_i (Z_{ii})_+, \qquad \zeta_- := \max_i(-Z_{ii})_+ ,
\end{equation}
for the difference between the two operators and its extreme entries. It is a diagonal matrix for all six metrics and is listed in Table \ref{tab:constants}. We also write
\begin{equation}\label{eqn:bT}
b_T := \min_{\eta\in T_{\bu^*}\mathcal M\setminus\{0\}}\frac{\ip{B^*\eta,\eta}_h}{\ip{\eta,\eta}_h}\ \geq\ \lambda_1(B^*), \qquad
\sigma_T := \min_{\eta\in T_{\bu^*}\mathcal M\setminus\{0\}}\frac{\ip{H^*\eta,\eta}_h}{\ip{\eta,\eta}_h},
\end{equation}
for the smallest eigenvalues of the quadratic forms of $B^*$ and of $H^*$ on the tangent space.

\begin{lem}[Coercivity of the constrained Hessian]\label{lem:coercive}
Let $\Lambda := \max\{\lambda_2(A_\beta),\lambda_1(A_{3\beta})\}$ and $\sigma := \Lambda - \lambda_1(A_\beta)$. Then $\sigma_T\geq\sigma>0$, and $b_T = \lambda_2(A_\beta)$ for the $a_u$ metric.
\end{lem}
\begin{proof}
The tangent space is the $\ip{\cdot,\cdot}_h$-orthogonal complement of $\bu^*$, the eigenvector of the simple smallest eigenvalue of $A_\beta$, so the Courant--Fischer principle in $\ip{\cdot,\cdot}_h$ gives $\ip{A_\beta\eta,\eta}_h\geq\lambda_2(A_\beta)\ip{\eta,\eta}_h$ on $T_{\bu^*}\mathcal M$, with equality for the second eigenvector, which is the claim on $b_T$. Writing $H^* = A_\beta - \lambda_1(A_\beta)\mathbb I + 2\beta\diag((\bu^*)^2)$ with $\diag((\bu^*)^2)\succeq0$ gives $\sigma_T\geq\lambda_2(A_\beta) - \lambda_1(A_\beta)$, while $H^*\succeq\lambda_1(A_{3\beta}) - \lambda_1(A_\beta)$ on all of $\mathbb R^N$, so $\sigma_T\geq\sigma$. Finally $\lambda_1(A_{3\beta})>\lambda_1(A_\beta)$. Let $\bv$ be a normalized eigenvector of $A_{3\beta}$ for $\lambda_1(A_{3\beta})$. Since $A_{3\beta}-A_\beta = 2\beta\diag((\bu^*)^2)\succeq0$, $\lambda_1(A_\beta)\leq\ip{A_\beta\bv,\bv}_h\leq\ip{A_{3\beta}\bv,\bv}_h = \lambda_1(A_{3\beta})$. If $\lambda_1(A_{3\beta}) = \lambda_1(A_\beta)$, then $\bv$ attains the minimum of the Rayleigh quotient of $A_\beta$, so $\bv = \pm\bu^*$ by simplicity, and then $\lambda_1(A_{3\beta}) = \ip{A_{3\beta}\bu^*,\bu^*}_h = \lambda_1(A_\beta) + 2\beta\ip{(\bu^*)^2,(\bu^*)^2}_h > \lambda_1(A_\beta)$, a contradiction. With $\lambda_2(A_\beta)>\lambda_1(A_\beta)$ this gives $\sigma>0$.
\end{proof}

\begin{lem}[Discrete maximum principle]\label{lem:max-principle}
Under Assumption \ref{asp:monotone}, $\beta\max_i(\bu^*_i)^2 = \beta(\bu^*_{i_0})^2\leq\lambda_1(A_\beta) - \mathbb V_{i_0}\leq\lambda_1(A_\beta)$, where $i_0$ is an index at which $\bu^*_i$ is maximal.
\end{lem}
\begin{proof}
Since $\bu^*>0$ by Lemma \ref{lem:perron}, $\mathbb S_{i_0j}\leq0$ and $\bu^*_j\leq\bu^*_{i_0}$ for $j\neq i_0$, and the row sums of $\mathbb S$ are nonnegative, $(\mathbb S\bu^*)_{i_0}\geq(\sum_j\mathbb S_{i_0j})\bu^*_{i_0}\geq0$, hence $(-\Delta_h\bu^*)_{i_0}\geq0$, and the $i_0$-th component of $A_\beta\bu^* = \lambda_1(A_\beta)\bu^*$ gives $\beta(\bu^*_{i_0})^2\leq\lambda_1(A_\beta) - \mathbb V_{i_0}$, with $\mathbb V\geq0$.
\end{proof}

\begin{thm}[Two-sided bounds of the Rayleigh quotient]\label{thm:rayleigh-quotient-bounds}
Let $\bu^*$ be the ground state. For every metric and every $\eta\in T_{\bu^*}\mathcal M\setminus\{0\}$,
\begin{equation}\label{eqn:two-sided}
\frac{\sigma_T}{\sigma_T+\zeta_-}\ \leq\ \rho_{\bu^*}(\eta)\ \leq\ 1+\frac{\zeta_+}{b_T}, \qquad\text{thus}\qquad
\kappa(\bu^*)\ \leq\ \Big(1+\frac{\zeta_+}{b_T}\Big)\Big(1+\frac{\zeta_-}{\sigma}\Big)
\end{equation}
for the condition number $\kappa(\bu^*)$ of $\Hess E(\bu^*)$, the ratio of its largest to its smallest eigenvalue.
\end{thm}
\begin{proof}
Since $Z$ and $\mathbb M$ are diagonal and $\mathbb M$ is positive, $-\zeta_-\ip{\eta,\eta}_h\leq\ip{Z\eta,\eta}_h\leq\zeta_+\ip{\eta,\eta}_h$. On $T_{\bu^*}\mathcal M$ therefore $H^* = B^*+Z\preceq B^*+\zeta_+\mathbb I\preceq(1+\zeta_+/b_T)B^*$ and $B^* = H^*-Z\preceq H^*+\zeta_-\mathbb I\preceq(1+\zeta_-/\sigma_T)H^*$ as quadratic forms, which by \eqref{eqn:rho-hweighted} are the two bounds in \eqref{eqn:two-sided}. The bound on $\kappa(\bu^*)$ follows from $\sigma_T\geq\sigma$.
\end{proof}

\begin{table}[t]
\centering\footnotesize
\setlength{\tabcolsep}{4pt}
\begin{tabular}{lllll}
\hline
metric & $Z = H^* - B^*$ & $\zeta_-\leq$ & $b_T\geq$ & $\zeta_+\leq$ (Assumption \ref{asp:monotone})\\
\hline
modified $H^1$ & $\mathbb V + 3\beta\diag((\bu^*)^2) - (\lambda_1+\alpha)\mathbb I$ & $\lambda_1+\alpha$ & $\lambda_1(-\Delta_h)+\alpha$ & $(\max_i\mathbb V_{i} + 2\lambda_1 - \alpha)_+$\\
$a_0$ & $3\beta\diag((\bu^*)^2) - \lambda_1\mathbb I$ & $\lambda_1$ & $\lambda_1(-\Delta_h+\mathbb V)$ & $2\lambda_1$\\
$a_u$ & $2\beta\diag((\bu^*)^2) - \lambda_1\mathbb I$ & $\lambda_1$ & $\lambda_2(A_\beta)$ & $\lambda_1$\\
\hline
modified $a_0$ & $\mathbb V_2 + 3\beta\diag((\bu^*)^2) - \lambda_1\mathbb I$ & $\lambda_1+\mathcal V_-$ & $\lambda_1(A_1)$ & $(\mathcal V_+ + 2\lambda_1)_+$\\
modified $a_u$ & $-\lambda_1\mathbb I$ & $\lambda_1$ & $\lambda_1+\sigma_T$ & $0$\\
Thomas--Fermi & $Z_{\mathrm{TF}}$ of \eqref{eqn:ZTF} & $e_{\mathrm{TF}}+\ell\lambda_1^{1/3}$ & $\ell\lambda_1^{1/3}$ & $(e_{\mathrm{TF}}-\ell\lambda_1^{1/3})_+$\\
\hline
\end{tabular}
\caption{The difference $Z = H^* - B^*$ at the ground state and bounds on $\zeta_\pm$ of \eqref{eqn:defect} and $b_T$ of \eqref{eqn:bT}. Here $\lambda_1 = \lambda_1(A_\beta)$, $\mathcal V_\pm := \max_i(\pm(\mathbb V_2)_{ii})_+$ and $e_{\mathrm{TF}}$ is the Thomas--Fermi error \eqref{eqn:tf-defect}. The last column uses Lemma \ref{lem:max-principle}, and the Thomas--Fermi entries hold without Assumption \ref{asp:monotone}. The entries $b_T$ of the $a_u$ and modified $a_u$ metrics and $\zeta_\pm$ of the modified $a_u$ metric are exact.}
\label{tab:constants}
\end{table}

\begin{coro}[Condition numbers of the classical metrics]\label{coro:conditioning}
Let $\bu^*$ be the ground state, let $\Lambda$ be as in Lemma \ref{lem:coercive}, and let Assumption \ref{asp:monotone} hold. Then
\begin{align*}
\kappa_{a_u}(\bu^*) &\leq \Big(1 + \frac{\zeta_+}{\lambda_2(A_\beta)}\Big)\frac{\Lambda}{\Lambda - \lambda_1(A_\beta)} \leq \Big(1 + \frac{\lambda_1(A_\beta)}{\lambda_2(A_\beta)}\Big)\frac{\Lambda}{\Lambda - \lambda_1(A_\beta)} \leq \frac{2\,\lambda_1(A_{3\beta})}{\lambda_1(A_{3\beta}) - \lambda_1(A_\beta)}, \\
\kappa_{a_0}(\bu^*) &\leq \Big(1 + \frac{2\lambda_1(A_\beta)}{\lambda_1(-\Delta_h+\mathbb V)}\Big)\frac{\Lambda}{\Lambda - \lambda_1(A_\beta)}, \\
\kappa_{H^1}(\bu^*) &\leq \Big(1 + \frac{(\max_i \mathbb V_i + 2\lambda_1(A_\beta) - \alpha)_+}{\lambda_1(-\Delta_h)+\alpha}\Big)\frac{\Lambda + \alpha}{\Lambda - \lambda_1(A_\beta)},
\end{align*}
and for the $a_u$ metric $1 - \lambda_1(A_\beta)/\Lambda\leq\rho_{\bu^*}(\eta)\leq1+\lambda_1(A_\beta)/\lambda_2(A_\beta)<2$ on $T_{\bu^*}\mathcal M\setminus\{0\}$, with $\zeta_+ = 0$ and the spectrum in $[1-\lambda_1(A_\beta)/\Lambda,\ 1]$ as soon as $2\beta\max_i(\bu^*_i)^2\leq\lambda_1(A_\beta)$. Without Assumption \ref{asp:monotone}, the same bounds hold with $\zeta_+ = \max_i(Z_{ii})_+$ of Table \ref{tab:constants} in place of the numerators $\lambda_1(A_\beta)$, $2\lambda_1(A_\beta)$ and $(\max_i\mathbb V_i+2\lambda_1(A_\beta)-\alpha)_+$, and the lower bound $1-\lambda_1(A_\beta)/\Lambda$ is unchanged. In particular, the bound $\rho_{\bu^*}(\eta)<2$ for the $a_u$ metric becomes $\rho_{\bu^*}(\eta)\leq1+\zeta_+/\lambda_2(A_\beta)$, and the factor $2$ in the last bound on $\kappa_{a_u}(\bu^*)$ becomes $1+\zeta_+/\lambda_2(A_\beta)$.
\end{coro}
\begin{proof}
The matrices $Z$ in Table \ref{tab:constants} are read off \eqref{eqn:six-metrics}. The bounds on $\zeta_-$ follow from $(\bu^*_i)^2\geq0$ and $\mathbb V\geq0$, the bounds on $b_T$ follow from \eqref{eqn:bT} and Lemma \ref{lem:coercive}, and the bounds on $\zeta_+$ follow from Lemma \ref{lem:max-principle}, for instance $2\beta(\bu^*_i)^2 - \lambda_1(A_\beta)\leq\lambda_1(A_\beta) - 2\mathbb V_{i_0}\leq\lambda_1(A_\beta)$ and $3\beta(\bu^*_i)^2 - \lambda_1(A_\beta)\leq2\lambda_1(A_\beta)$. Theorem \ref{thm:rayleigh-quotient-bounds} gives the bounds, with $\sigma/(\sigma+\lambda_1(A_\beta)) = 1-\lambda_1(A_\beta)/\Lambda$ and $1+(\lambda_1(A_\beta)+\alpha)/\sigma = (\Lambda+\alpha)/(\Lambda-\lambda_1(A_\beta))$ by $\sigma = \Lambda-\lambda_1(A_\beta)$. The last inequality for $\kappa_{a_u}(\bu^*)$ follows from $1+\lambda_1(A_\beta)/\lambda_2(A_\beta)<2$ and from the fact that $\Lambda/(\Lambda-\lambda_1(A_\beta))$ decreases in $\Lambda\geq\lambda_1(A_{3\beta})$.
\end{proof}

For the $a_u$ metric in the harmonic trap $V = |\bx|^2$, Appendix \ref{app:spectra} compares these bounds with the exact extreme eigenvalues of $\Hess E(\bu^*)$ in one, two and three dimensions. For $\beta\geq10^2$, the lower bound $1-\lambda_1(A_\beta)/\Lambda$ is within $9\%$ of the exact minimum in one dimension and within $21\%$ in two and three dimensions. The bound $1-\lambda_1(A_\beta)/\lambda_2(A_\beta)$ based on the spectral gap alone is smaller than the exact minimum by a factor up to $3\times10^4$.

\subsection{The modified $a_0$ metric}\label{sec:modified-a0}
Consider a splitting $V = V_1+V_2$ as in Section \ref{sec:classical-metrics}, where $V_1\geq0$ is separable, $V_1(\bx) = \sum_{j=1}^d v_j(x_j)$, and $V_2$ is bounded. Then the operator $A_1 = -\Delta_h+\mathbb V_1$ can be inverted exactly and efficiently on GPU by tensor-product direct solvers \citep{liu2026gpu}. When $V_1$ is confining, that is, $V_1(\bx)\to\infty$ as $|\bx|\to\infty$, the preconditioned operator $A_1^{-1}(-\Delta_h+\mathbb V) = \mathbb I+A_1^{-1}\mathbb V_2$ has eigenvalues clustered at one and a condition number bounded independently of the mesh size and the domain size \citep[Theorems~3.4 and~3.5]{liu2026gpu}. This motivates taking $A_1$ as the metric operator itself, which defines the modified $a_0$ metric \eqref{eqn:six-metrics}. By \eqref{eqn:rho-hweighted} the Rayleigh quotient of the Riemannian Hessian for this metric at the ground state is $1+\ip{Z\eta,\eta}_h/\ip{A_1\eta,\eta}_h$ with $Z = \mathbb V_2+3\beta\diag((\bu^*)^2)-\lambda_1(A_\beta)\mathbb I$, which is the matrix $Z$ of the $a_0$ metric plus the nonseparable part of the potential. 

\begin{prop}[Condition number of the modified $a_0$ metric]\label{prop:modified-a0}
Let $\bu^*$ be the ground state, $\Lambda$ as in Lemma \ref{lem:coercive} and $\mathcal V_\pm = \max_i(\pm(\mathbb V_2)_{ii})_+$. Under the modified $a_0$ metric, denoted by $\widetilde a_0$ in subscripts, for all $\eta\in T_{\bu^*}\mathcal M\setminus\{0\}$,
\[
\frac{\Lambda-\lambda_1(A_\beta)}{\Lambda+\mathcal V_-}\ \leq\ \rho_{\bu^*}(\eta)\ \leq\ 1+\frac{\max_i\big((\mathbb V_2)_{ii}+3\beta(\bu^*_i)^2-\lambda_1(A_\beta)\big)_+}{\lambda_1(A_1)},
\]
and under Assumption \ref{asp:monotone}
\[
\kappa_{\widetilde a_0}(\bu^*)\ \leq\ \Big(1+\frac{(\mathcal V_++2\lambda_1(A_\beta))_+}{\lambda_1(A_1)}\Big)\frac{\Lambda+\mathcal V_-}{\Lambda-\lambda_1(A_\beta)}.
\]
If $V_2\geq0$, then $\kappa_{\widetilde a_0}(\bu^*)\leq\big(1+\mathcal V_+/\lambda_1(A_1)\big)\,\kappa_{a_0}(\bu^*)$.
\end{prop}
\begin{proof}
The entries of $Z = \mathbb V_2+3\beta\diag((\bu^*)^2)-\lambda_1(A_\beta)\mathbb I$ are at least $-\mathcal V_--\lambda_1(A_\beta)$ and, for a monotone scheme, at most $\mathcal V_++2\lambda_1(A_\beta)$ by Lemma \ref{lem:max-principle}, and $b_T\geq\lambda_1(A_1)$ by \eqref{eqn:bT}, so Theorem \ref{thm:rayleigh-quotient-bounds} with $\sigma = \Lambda-\lambda_1(A_\beta)$ gives the first two claims. For the last claim, $-\Delta_h+\mathbb V = A_1+\mathbb V_2$ and $0\leq\mathbb V_2\leq\mathcal V_+\mathbb I$ give $\ip{A_1\bx,\bx}_h\leq\ip{(A_1+\mathbb V_2)\bx,\bx}_h\leq(1+\mathcal V_+/\lambda_1(A_1))\ip{A_1\bx,\bx}_h$, so by \eqref{eqn:rho-hweighted} the Rayleigh quotients of the two metrics differ pointwise by a factor between $1$ and $1+\mathcal V_+/\lambda_1(A_1)$, and so do their extremes.
\end{proof}

Compared with the bound on $\kappa_{a_0}(\bu^*)$ in Corollary \ref{coro:conditioning}, the bound in Proposition \ref{prop:modified-a0} replaces $\lambda_1(-\Delta_h+\mathbb V)$ by $\lambda_1(A_1)$ and adds the bounded quantities $\mathcal V_\pm$ to the numerators. Thus the modified $a_0$ metric keeps the conditioning of the $a_0$ metric up to a bounded factor. Unlike the modified $H^1$ metric, whose bound grows with the domain size, it is robust with respect to the domain size.

\begin{rem}[Domain robustness of the modified $a_0$ metric]\label{rem:domain-robust}
Let Assumption \ref{asp:monotone} hold, let $V_1$ be confining and $V_2$ be bounded on $\mathbb R^d$, and let the problem be posed on the boxes $\Omega_L = (-L,L)^d$ with a fixed mesh size $h$ and $L\in h\mathbb N$, so that the grids are nested. Then $\lambda_1(A_1)$ is bounded below uniformly in $L$, because the first eigenvalue of $-\Delta_h+\mathbb V_1$ on $\Omega_L$ is at least that on the infinite grid, by the min-max argument of \citet[proof of Theorem~3.5]{liu2026gpu}. By \eqref{eqn:lambda-u}, $\lambda_1(A_\beta)\leq4E(\bu^*)$, which is bounded above uniformly in $L$ by comparison with a fixed grid function supported in the smallest box. Thus, if the coercivity constant $\sigma = \Lambda-\lambda_1(A_\beta)$ is bounded below uniformly in $L$, then $\kappa_{\widetilde a_0}(\bu^*)$ is bounded uniformly in $L$ by Proposition \ref{prop:modified-a0}, since $(\Lambda+\mathcal V_-)/(\Lambda-\lambda_1(A_\beta)) = 1+(\lambda_1(A_\beta)+\mathcal V_-)/\sigma$. For the modified $H^1$ metric, the bound in Corollary \ref{coro:conditioning} grows with $L$, because $\lambda_1(-\Delta_h)\to0$ and $\max_i\mathbb V_i\to\infty$ as $L\to\infty$. In the tests of Section \ref{ssec:domain}, the iteration count of the modified $H^1$ flow grows with $L$ while that of the modified $a_0$ flow does not.
\end{rem}

\subsection{The modified $a_u$ metric}\label{sec:modified-au}
On the tangent space, $g_{\bu^*}(\Hess E(\bu^*)[\eta],\eta) = g_{\bu^*}(\eta,\eta)+\ip{Z\eta,\eta}_h$ by \eqref{eqn:rho-hweighted}, so the closer the metric operator is to the constrained Hessian $H^* = A_{3\beta}-\lambda_1(A_\beta)\mathbb I$, the closer the Riemannian Hessian is to the identity. A metric operator must be positive definite at every $\bu\in\mathbb R^N$. A natural choice is $A_{3\beta}(\bu)$ itself, which is positive definite for every $\bu$ and at $\bu^*$ differs from $H^*$ only by the scalar shift $\lambda_1(A_\beta)\mathbb I$. Including the shift, which amounts to Newton's method, destroys positive definiteness away from the ground state. The operator $A_{3\beta}(\bu)$ thus defines the modified $a_u$ metric \eqref{eqn:six-metrics}. At the ground state, $Z = -\lambda_1(A_\beta)\mathbb I$ is a scalar matrix, so the Riemannian Hessian is explicit.

\begin{prop}[Condition number of the modified $a_u$ metric]\label{prop:modified-au}
Let $\bu^*$ be the ground state and let $s_1\leq\cdots\leq s_{N-1}$ be the eigenvalues of the quadratic form $\ip{H^*\eta,\eta}_h$ on $T_{\bu^*}\mathcal M$, that is, of the $\ip{\cdot,\cdot}_h$-orthogonal compression of $H^*$ to the tangent space, so that $s_1 = \sigma_T$. Under the modified $a_u$ metric, denoted by $\widetilde a_u$ in subscripts, the eigenvalues of $\Hess E(\bu^*)$ are
\begin{equation}\label{eqn:modified-au-spectrum}
\rho_k = \frac{s_k}{s_k+\lambda_1(A_\beta)}, \qquad k = 1,\ldots,N-1,
\end{equation}
and thus, without any assumption on the scheme, $\sigma_T/(\lambda_1(A_\beta)+\sigma_T)\leq\rho_{\bu^*}(\eta)<1$ on $T_{\bu^*}\mathcal M\setminus\{0\}$ and
\[
\kappa_{\widetilde a_u}(\bu^*) = \frac{s_{N-1}}{s_{N-1}+\lambda_1(A_\beta)}\Big(1+\frac{\lambda_1(A_\beta)}{\sigma_T}\Big)\ \leq\ 1+\frac{\lambda_1(A_\beta)}{\sigma_T}.
\]
\end{prop}
\begin{proof}
Since $\Hess E(\bu^*)$ is self-adjoint with respect to $g_{\bu^*}$, polarization of \eqref{eqn:second-variation} \citep[(5.36)]{absil2008optimization} gives $g_{\bu^*}(\Hess E(\bu^*)[\eta],\xi) = \ip{H^*\eta,\xi}_h$ for all $\eta,\xi\in T_{\bu^*}\mathcal M$. Hence $\Hess E(\bu^*)[\eta] = \rho\eta$ if and only if $\ip{H^*\eta,\xi}_h = \rho\ip{B^*\eta,\xi}_h$ for all $\xi\in T_{\bu^*}\mathcal M$. With $B^* = A_{3\beta} = H^*+\lambda_1(A_\beta)\mathbb I$ this reads $(1-\rho)\ip{H^*\eta,\xi}_h = \rho\lambda_1(A_\beta)\ip{\eta,\xi}_h$. Here $\rho\neq1$, since $\rho = 1$ forces $\eta = 0$. Thus $\eta$ is an eigenvector of $\Hess E(\bu^*)$ with eigenvalue $\rho$ if and only if it is an eigenvector of the compression of $H^*$ to $T_{\bu^*}\mathcal M$ with eigenvalue $s = \rho\lambda_1(A_\beta)/(1-\rho)$, that is, $\rho = s/(s+\lambda_1(A_\beta))$.
\end{proof}

Unlike for the $a_u$ metric, $\nabla_XE(\bu) = \bu-2\beta A_{3\beta}(\bu)^{-1}\bu^3$ is not $\bu$, so one iteration requires two solves instead of one. Remark \ref{rem:tf-orders} compares this cost with the gain in conditioning in the Thomas--Fermi regime.

\begin{rem}[Family of metrics between $a_u$ and modified $a_u$]\label{rem:family}
The $a_u$ and modified $a_u$ metrics are the members $c = 1$ and $c = 3$ of the family of $\bu$-dependent metrics induced by
\begin{equation}\label{eqn:family}
B_c(\bu) := -\Delta_h+\mathbb V+c\beta\diag(\bu^2), \qquad c\geq1,
\end{equation}
for which $Z_c = (3-c)\beta\diag((\bu^*)^2)-\lambda_1(A_\beta)\mathbb I$ at the ground state. For $c\leq3$ the entries of $Z_c$ are at least $-\lambda_1(A_\beta)$, so $\zeta_-\leq\lambda_1(A_\beta)$, and under Assumption \ref{asp:monotone} they are at most $(2-c)\lambda_1(A_\beta)$ by Lemma \ref{lem:max-principle}, so $\zeta_+ = 0$ for every $c\in[2,3]$. Thus, under Assumption \ref{asp:monotone}, all members with $c\in[2,3]$ share the bound $\kappa_c(\bu^*)\leq1+\lambda_1(A_\beta)/\sigma_T$ of Proposition \ref{prop:modified-au}. Only $c = 3$ has a scalar $Z_c$ and thus the explicit spectrum \eqref{eqn:modified-au-spectrum}, and for this member the bound holds without any assumption on the scheme (Proposition \ref{prop:modified-au}).
\end{rem}

\subsection{The Thomas--Fermi metric}\label{sec:tf-metric}
We now explain how the new Thomas--Fermi metric is designed. By the Rayleigh quotient \eqref{eqn:rho-hweighted}, the metric operator should approximate the constrained Hessian $H^* = A_{3\beta}-\lambda_1\mathbb I$ on the tangent space, where $\lambda_1:=\lambda_1(A_\beta)$. The operator $A_{3\beta}(\bu)-\lambda(\bu)\mathbb I$ at a general iterate $\bu$ is not positive definite (Section \ref{sec:modified-au}), so $H^*$ cannot be used directly. Inspired by the Thomas--Fermi approximation $\beta(\bu^*)^2\approx(\lambda_1-\mathbb V)_+$, if we write $\mathbb V-\lambda_1\mathbb I = (\mathbb V-\lambda_1)_+-(\lambda_1-\mathbb V)_+$, then
\begin{equation}\label{eqn:H-split}
H^* = \underbrace{-\Delta_h+2\beta\diag((\bu^*)^2)+\diag\big((\mathbb V-\lambda_1)_+\big)}_{\succ\,0}
    +\diag\big(\beta(\bu^*)^2-(\lambda_1-\mathbb V)_+\big),
\end{equation}
where the last term is the error of the Thomas--Fermi approximation. The first term is positive definite, and it remains so at a general iterate $\bu$ with $\lambda(\bu)$ of \eqref{eqn:lambda-u} in place of $\lambda_1$, because each of its three terms is positive semidefinite. Replacing the Thomas--Fermi approximation error by $\ell\lambda(\bu)^{1/3}\,\mathbb I$ gives the Thomas--Fermi metric operator $B_{\mathrm{TF}}(\bu)$ of \eqref{eqn:tf-metric}, where the power $1/3$ is taken according to the analysis in \citet{gallo2011thomas} and Corollary \ref{thm:TF} below. To our knowledge, the metric \eqref{eqn:tf-metric} has not previously been proposed. The Thomas--Fermi preconditioner of \citet{antoine2014robust}, used alone or combined with a Laplace preconditioner in \citet{antoine2017efficient}, is the diagonal operator $\alpha+\mathbb V+\beta\bu^2$ obtained by dropping the kinetic term as in the Thomas--Fermi approximation, whereas the metric here keeps the Laplacian.

At the ground state, $\lambda(\bu^*) = \lambda_1$ and the matrix $Z$ of the metric is the Thomas--Fermi error shifted down by $\ell\lambda_1^{1/3}$,
\begin{equation}\label{eqn:ZTF}
Z_{\mathrm{TF}} = \diag\big(\beta(\bu^*)^2-(\lambda_1-\mathbb V)_+\big)-\ell\lambda_1^{1/3}\,\mathbb I.
\end{equation}
\begin{prop}[Condition number of the Thomas--Fermi metric]\label{prop:tf-metric}
Let $\bu^*$ be the ground state, $\lambda_1=\lambda_1(A_\beta)$ and
\begin{equation}\label{eqn:tf-defect}
e_{\mathrm{TF}}:=\max_i\big|\beta(\bu^*_i)^2-(\lambda_1-\mathbb V_i)_+\big|.
\end{equation}
Under the Thomas--Fermi metric with any $\ell>0$, denoted by $\mathrm{TF}$ in subscripts, the spectrum of $\Hess E(\bu^*)$ lies in $[\sigma_T/(\sigma_T+e_{\mathrm{TF}}+\ell\lambda_1^{1/3}),\,\max\{1,\,e_{\mathrm{TF}}/(\ell\lambda_1^{1/3})\}]$, and thus
\begin{equation}\label{eqn:tf-kappa}
\kappa_{\mathrm{TF}}(\bu^*)\leq\max\Big\{1,\frac{e_{\mathrm{TF}}}{\ell\lambda_1^{1/3}}\Big\}\Big(1+\frac{e_{\mathrm{TF}}+\ell\lambda_1^{1/3}}{\sigma_T}\Big).
\end{equation}
In particular, if $e_{\mathrm{TF}}\leq C_{\mathrm{TF}}\lambda_1^{1/3}$ and $\sigma_T\geq c_\sigma\lambda_1^{1/3}$ with constants independent of $\beta$, then
\begin{equation}\label{eqn:tf-uniform}
\kappa_{\mathrm{TF}}(\bu^*)\leq\max\Big\{1,\frac{C_{\mathrm{TF}}}{\ell}\Big\}\Big(1+\frac{C_{\mathrm{TF}}+\ell}{c_\sigma}\Big)
\end{equation}
uniformly in $\beta$, for every $\ell>0$.
\end{prop}
\begin{proof}
Every diagonal entry of \eqref{eqn:ZTF} lies in $[-e_{\mathrm{TF}}-\ell\lambda_1^{1/3},\,e_{\mathrm{TF}}-\ell\lambda_1^{1/3}]$, so $\zeta_-\leq e_{\mathrm{TF}}+\ell\lambda_1^{1/3}$ and $\zeta_+\leq(e_{\mathrm{TF}}-\ell\lambda_1^{1/3})_+$, and $b_T\geq\lambda_1(B^*)\geq\ell\lambda_1^{1/3}$ by \eqref{eqn:bT}, because $B_{\mathrm{TF}}(\bu^*)-\ell\lambda_1^{1/3}\mathbb I$ is positive semidefinite. Theorem \ref{thm:rayleigh-quotient-bounds} gives the spectrum, since $1+\zeta_+/b_T\leq1+(e_{\mathrm{TF}}/(\ell\lambda_1^{1/3})-1)_+ = \max\{1,e_{\mathrm{TF}}/(\ell\lambda_1^{1/3})\}$, and the last claim follows by inserting the stated hypotheses.
\end{proof}

For strong interaction, the following result, which is a simple corollary of the results in \citet{gallo2011thomas}, gives the sharp deterioration of the $a_u$ metric for the one-dimensional harmonic trap.

\begin{coro}[Thomas--Fermi asymptotics of the coercivity constant]\label{thm:TF}
For the continuous problem on $\Omega = \mathbb R$ with $V(x) = x^2$, let $u$ be the positive minimizer of $\mathcal E$ on $\{u\in H^1(\mathbb R): \int_{\mathbb R}x^2u^2\,\ud x<\infty,\ \norm{u}_{L^2(\mathbb R)} = 1\}$, let $A_{c\beta} := -\partial_x^2+x^2+c\beta u^2$ on $L^2(\mathbb R)$, $T_u := \{\eta\in H^1(\mathbb R): x\eta\in L^2(\mathbb R),\ (\eta,u)=0\}$, and let $\lambda_1 = \lambda_1(A_\beta)$ be the eigenvalue of $u$, which tends to infinity as $\beta\to\infty$. Then, as $\beta\to\infty$,
\begin{equation}\label{eqn:TF-sigma}
\lambda_1(A_{3\beta}) - \lambda_1(A_\beta) = \mu_1\,\lambda_1^{1/3}\,\big(1+o(1)\big), \qquad \mu_1 = 2.4105\ldots,
\end{equation}
where $\mu_1$ is the smallest eigenvalue of the Schr\"odinger operator $M_0 = -4\partial_y^2 + 3\nu_0(y)^2 - y$ on $L^2(\mathbb R)$ and $\nu_0$ is the Hastings--McLeod solution of the Painlev\'e-II equation $4\nu''+y\nu-\nu^3 = 0$, characterized by $\nu_0(y)\sim\sqrt y$ as $y\to+\infty$ and $\nu_0(y)\to0$ as $y\to-\infty$. Thus, for the $a_u$ metric,
\begin{equation}\label{eqn:TF-scaling}
\min_{T_u}\rho_u = \mu_1\,\lambda_1^{-2/3}\big(1+o(1)\big), \qquad \max_{T_u}\rho_u\to2, \qquad \kappa_{a_u}(u) = \frac{2}{\mu_1}\,\lambda_1^{2/3}\big(1+o(1)\big),
\end{equation}
and the lower bound $1-\lambda_1(A_\beta)/\Lambda$ of Corollary \ref{coro:conditioning} is asymptotically exact.
\end{coro}
\begin{proof}
Equation \eqref{eqn:TF-sigma} is \citet[Theorem~2]{gallo2011thomas} in our variables. Set $x=\sqrt{\lambda_1}y$ and $u(x)=\sqrt{\lambda_1/\beta}\,U(y)$. The ground state equation becomes $\varepsilon^2U''+(1-y^2)U-U^3=0$ with $\varepsilon=\lambda_1^{-1}$, and the continuous constrained Hessian $H := A_{3\beta}-\lambda_1$ satisfies $\lambda_1^{-1}H=-\varepsilon^2\partial_y^2+3U^2-1+y^2=L_+^\varepsilon$ in the notation of \citet{gallo2011thomas}. By \citet[Theorem~2]{gallo2011thomas}, the first two eigenvalues $\delta_1\leq\delta_2$ of $H$ are
\[
\delta_j=\mu_1\lambda_1^{1/3}(1+o(1)),\qquad j=1,2.
\]
These two eigenvalues correspond to the two boundary layers, and $\mu_1>0$ by \citet[Lemma~2.2]{gallo2011thomas}. Since $\delta_1=\lambda_1(A_{3\beta})-\lambda_1$, this proves \eqref{eqn:TF-sigma}.

As quadratic forms, $\lambda_1\mathbb I\preceq A_\beta\preceq H+\lambda_1\mathbb I$, so $\ip{H\eta,\eta}/\ip{A_\beta\eta,\eta}\geq\ip{H\eta,\eta}/(\ip{H\eta,\eta}+\lambda_1\norm\eta^2)\geq\delta_1/(\delta_1+\lambda_1)$ on $T_u$, because $t\mapsto t/(t+\lambda_1)$ is increasing and $\ip{H\eta,\eta}\geq\delta_1\norm\eta^2$. The second eigenfunction $v_2$ of the even Sturm--Liouville operator $H$ is odd, hence orthogonal to $u$, and it gives the upper bound in
\[
\frac{\delta_1}{\delta_1+\lambda_1}
\leq\min_{T_u}\frac{\ip{H\eta,\eta}}{\ip{A_\beta\eta,\eta}}
\leq\frac{\ip{Hv_2,v_2}}{\lambda_1\norm{v_2}^2}
=\frac{\delta_2}{\lambda_1}.
\]
Both endpoints are $\mu_1\lambda_1^{-2/3}(1+o(1))$. The bound $\sigma/(\sigma+\lambda_1)=1-\lambda_1/\Lambda$ lies between the left endpoint and the minimum, since $\sigma\geq\delta_1$, so it too is asymptotically exact. Finally, $\rho_u(\eta) = 1+\ip{(2\beta u^2-\lambda_1)\eta,\eta}/\ip{A_\beta\eta,\eta}\leq2$, because $\beta u^2\leq\lambda_1$ by the maximum principle at the maximum of $u$ and $A_\beta\succeq\lambda_1\mathbb I$. A fixed odd bump $\varphi$ supported near the origin gives $\rho_u(\varphi)\to2$, because $\beta u^2=\lambda_1-x^2+o(\lambda_1)$ uniformly there by \citet[Theorem~1]{gallo2011thomas}. This proves \eqref{eqn:TF-scaling}.
\end{proof}

\begin{rem}[Condition numbers in the Thomas--Fermi regime]\label{rem:tf-orders}
For the continuous one-dimensional harmonic problem of Corollary \ref{thm:TF}, the condition numbers behave as follows as $\beta\to\infty$, with $\lambda_1 = \lambda_1(A_\beta)\to\infty$.
\begin{itemize}
\item \emph{The $a_u$ metric.} By \eqref{eqn:TF-scaling}, $\kappa_{a_u}(u) = (2/\mu_1)\lambda_1^{2/3}(1+o(1))$ grows like $\lambda_1^{2/3}$.
\item \emph{The family \eqref{eqn:family}.} Since $A_\beta\preceq B_c\preceq cA_\beta$ for $c\geq1$, \eqref{eqn:rho-hweighted} gives $\kappa_c(u)\geq\kappa_{a_u}(u)/c$, so every member grows at least like $\lambda_1^{2/3}$. Its member $c = 3$, the modified $a_u$ metric, has $\kappa_{\widetilde a_u}(u) = 1+\lambda_1/\sigma_T\sim\mu_1^{-1}\lambda_1^{2/3}$ by Proposition \ref{prop:modified-au} and \eqref{eqn:TF-sigma}, since the Rayleigh quotient $s/(s+\lambda_1)$ in \eqref{eqn:modified-au-spectrum} has supremum $1$ over the unbounded spectrum of $H$ on $T_u$, approached by high modes and not attained. This is half of $\kappa_{a_u}(u)$, but each iteration needs two solves instead of one (Section \ref{sec:modified-au}).
\item \emph{The $a_0$ metric.} Corollary \ref{coro:conditioning} with $\lambda_1(-\partial_x^2+x^2) = 1$ and $\sigma\geq\lambda_1(A_{3\beta})-\lambda_1$ gives $\kappa_{a_0}(u)\leq(1+2\lambda_1)\big(1+\lambda_1/(\lambda_1(A_{3\beta})-\lambda_1)\big) = (2/\mu_1)\lambda_1^{5/3}(1+o(1))$, so the bound grows like $\lambda_1^{5/3}$.
\item \emph{The Thomas--Fermi metric.} Its condition number is bounded uniformly in $\beta$ for every $\ell>0$, because both hypotheses of \eqref{eqn:tf-uniform} hold for large $\beta$. The two eigenvalues in the proof of Corollary \ref{thm:TF} enclose $\sigma_T$, so $\sigma_T = \mu_1\lambda_1^{1/3}(1+o(1))$ and any $c_\sigma<\mu_1$ works. The uniform approximation of $\beta u^2$ by $\lambda_1^{1/3}\nu_0(y)^2$ in the layer variable $y = (\lambda_1-x^2)/\lambda_1^{1/3}$ \citep[Theorem~1 and (3.17)]{gallo2011thomas} gives $e_{\mathrm{TF}} = (C_0+o(1))\lambda_1^{1/3}$ with $C_0 = \sup_y|\nu_0(y)^2-y_+|$, which is finite by \citet[Proposition~2.1]{gallo2011thomas}, so any $C_{\mathrm{TF}}>C_0$ works. Numerically $C_0 = \nu_0(0)^2\approx0.43$, so $C_{\mathrm{TF}}$ can be taken in $(C_0,1/2)$, and for $\ell = 1/2$ the first factor in \eqref{eqn:tf-uniform} equals one. Unlike the $a_u$ metric and its family, the Thomas--Fermi metric is not uniformly equivalent to $A_\beta$, since on the modes that attain $\sigma_T$ its quadratic form is of order $\lambda_1^{1/3}$ while that of $A_\beta$ is of order $\lambda_1$.
\end{itemize}
\end{rem}

Appendix \ref{app:spectra} tests these orders, in particular the uniform bound for the Thomas--Fermi metric, against exact spectra of the Riemannian Hessian in one and three dimensions.

\section{Convergence}\label{sec:convergence}
The results of this section are standard in form. A local linear rate determined by the spectrum of the Riemannian Hessian and global convergence to a critical point are classical for gradient descent on manifolds \citep{absil2008optimization,boumal2023intromanifolds,schneider2015convergence}. The point of this section is a uniform treatment of all the Sobolev gradient flows of Table \ref{tab:metric-list} under the same hypotheses. For instance, it is not clear whether convergence results for one specific flow, such as those for the discrete modified $H^1$ flow in \citet{chen2024fully}, extend to the Thomas--Fermi metric, whose operator is continuous but not smooth in $\bu$. Throughout this section we fix one of the six metrics, write $\norm{\eta}_\bu := g_\bu(\eta,\eta)^{1/2}$ for the metric norm, and use the following uniform bounds.

\begin{lem}[Uniform bounds on the metric and the energy]\label{lem:metric-equiv}
For a fixed discretization and fixed parameters, there are constants $0<c_1\leq c_2<\infty$ such that $c_1\norm{\bx}_h^2\leq\norm{\bx}_\bu^2\leq c_2\norm{\bx}_h^2$ for all $\bu\in\mathcal M$ and $\bx\in\mathbb R^N$. Thus $K := \{(\bu,\eta): \bu\in\mathcal M,\ \eta\in T_\bu\mathcal M,\ \norm\eta_\bu = 1\}$ is compact, and $C_g := \sup_{\bu\in\mathcal M}\norm{\grad E(\bu)}_\bu$ and $C_E := \max\{1,\ \sup\{|\frac{d^2}{dt^2}E(R_\bu(t\eta))|: (\bu,\eta)\in K,\ |t|\leq C_g\}\}$ are finite.
\end{lem}
\begin{proof}
The function $(\bu,\bx)\mapsto g_\bu(\bx,\bx)$ is positive and continuous on the compact set $\mathcal M\times\{\bx:\norm{\bx}_h=1\}$. For the Thomas--Fermi metric, the continuity holds because $\lambda(\bu)$ is continuous. We take $c_1$ and $c_2$ to be the minimum and the maximum of this function. The set $K$ is closed. Since $\norm\eta_h\leq c_1^{-1/2}$ for $(\bu,\eta)\in K$, the set $K$ is also bounded and therefore compact. The inverse $B(\bu)^{-1}$ is positive definite and depends continuously on $\bu$. By \eqref{eqn:projection}, $\grad E$ is therefore continuous on $\mathcal M$. Since $\mathcal M$ is compact and the metric is continuous, $C_g$ is finite.

For $(\bu,\eta)\in K$, we have $R_\bu(t\eta) = (\bu+t\eta)/(1+t^2\norm\eta_h^2)^{1/2}$ by \eqref{eqn:iteration}, and the right-hand side is a smooth function of $(\bu,\eta,t)\in\mathbb R^N\times\mathbb R^N\times\mathbb R$. Since $E$ is a polynomial, the second derivative of $E(R_\bu(t\eta))$ in $t$ is therefore a continuous function of $(\bu,\eta,t)$. Hence $C_E$ is finite, because $K\times[-C_g,C_g]$ is compact. Similar bounds for smooth metrics are given in \citet[Lemma~10.57]{boumal2023intromanifolds}.
\end{proof}

\subsection{Local convergence via the Riemannian Hessian}\label{sec:local}
Let $\bu^*$ be the ground state. As in Section \ref{sec:bounds}, we assume that $\bu^*$ is an eigenvector of the simple smallest eigenvalue $\lambda_1(A_\beta)$ of $A_\beta$, which holds under Assumption \ref{asp:monotone}. For the Thomas--Fermi metric, Theorem \ref{thm:asymptotic-rate} also assumes that no node satisfies $\mathbb V_i = \lambda_1(A_\beta)$, so that the metric is smooth near $\bu^*$ (Appendix \ref{app:riemannian-connection}). Proposition \ref{thm:local-convergence} does not need this assumption. By \eqref{eqn:rho-hweighted} and Lemma \ref{lem:coercive}, $\Hess E(\bu^*)$ is positive definite. We denote the eigenvalues of $\Hess E(\bu^*)$ by $0<\rho_1\leq\cdots\leq\rho_{N-1}$. Section \ref{sec:rayleigh} bounds the condition number $\kappa(\bu^*) = \rho_{N-1}/\rho_1$ for each metric. The local rate of the iteration \eqref{eqn:iteration} with a fixed step is controlled by this spectrum. \citet[Theorem~4.20]{boumal2023intromanifolds} proves local linear convergence for every fixed step $\tau<2/\rho_{N-1}$. Theorem \ref{thm:asymptotic-rate} gives the rate in terms of the spectrum and identifies the optimal step.

\begin{thm}[Local rate of the fixed-step iteration]\label{thm:asymptotic-rate}
For $\tau>0$ let $F_\tau(\bu) := R_\bu(-\tau\grad E(\bu))$, so that \eqref{eqn:iteration} with the constant step $\tau$ reads $\bu_{k+1} = F_\tau(\bu_k)$. Then $F_\tau$ is smooth near $\bu^*$, $F_\tau(\bu^*) = \bu^*$, and
\begin{equation}\label{eqn:DF}
\D F_\tau(\bu^*)[\xi] = \xi-\tau\,\Hess E(\bu^*)[\xi] \qquad\text{for all } \xi\in T_{\bu^*}\mathcal M .
\end{equation}
Thus, if $0<\tau<2/\rho_{N-1}$, then for every $\epsilon>0$ there are a neighborhood $S_\epsilon$ of $\bu^*$ in $\mathcal M$ and a constant $C_\epsilon$ such that for every $\bu_0\in S_\epsilon$
\begin{equation}\label{eqn:ostrowski-rate}
\norm{\bu_k-\bu^*}_h\leq C_\epsilon\,(r_\tau+\epsilon)^k\,\norm{\bu_0-\bu^*}_h, \qquad r_\tau := \max_i|1-\tau\rho_i| = \max\{1-\tau\rho_1,\ \tau\rho_{N-1}-1\}.
\end{equation}
The minimum of $r_\tau$ over $\tau$ is attained at $\tau^* = 2/(\rho_1+\rho_{N-1})$ and equals $(\kappa(\bu^*)-1)/(\kappa(\bu^*)+1)$.
\end{thm}
\begin{proof}
The map $F_\tau$ is smooth near $\bu^*$ because $\grad E$ and the retraction are smooth, and $F_\tau(\bu^*) = R_{\bu^*}(0) = \bu^*$. Let $\bu(s)$ be a smooth curve on $\mathcal M$ with $\bu(0) = \bu^*$ and $\bu'(0) = \xi$, and let $\mathbf p(s) := \grad E(\bu(s))$, viewed as a map into $\mathbb R^N$. Then $\mathbf p(0) = 0$ and $\mathbf p'(0) = \D\grad E(\bu^*)[\xi]$. Since $\mathbf p(s)\in T_{\bu(s)}\mathcal M$, the definition \eqref{eqn:iteration} gives
\[
F_\tau(\bu(s)) = \frac{\bu(s)-\tau\mathbf p(s)}{d(s)}, \qquad d(s) := \sqrt{1+\tau^2\norm{\mathbf p(s)}_h^2},
\]
with $d(0) = 1$ and $d'(0) = 0$ because $\mathbf p(0) = 0$. Differentiating at $s = 0$ gives $\D F_\tau(\bu^*)[\xi] = \xi-\tau\mathbf p'(0) = \xi-\tau\,\D\grad E(\bu^*)[\xi]$. Differentiating $\ip{\mathbf p(s),\bu(s)}_h = 0$ at $s = 0$ gives $\ip{\mathbf p'(0),\bu^*}_h = 0$, so $\D\grad E(\bu^*)[\xi]$ lies in $T_{\bu^*}\mathcal M$. By Lemma \ref{lem:riemannian-connection-ambient}, $\Hess E(\bu^*)[\xi] = \nabla_\xi\grad E = P_{\bu^*}\big(\D\grad E(\bu^*)[\xi]+\Gamma_{\bu^*}(\xi,\grad E(\bu^*))\big)$. Since $\grad E(\bu^*) = 0$ and $\D\grad E(\bu^*)[\xi]$ is tangent, this gives $\Hess E(\bu^*)[\xi] = \D\grad E(\bu^*)[\xi]$, which is \eqref{eqn:DF}. This also follows from \citet[Prop.~5.3]{boumal2023intromanifolds}. Now let $\Phi := R_{\bu^*}$, a smooth bijection from $T_{\bu^*}\mathcal M$ onto the open hemisphere around $\bu^*$ with smooth inverse (Section \ref{sec:submanifold}), and $\tilde F := \Phi^{-1}\circ F_\tau\circ\Phi$ near $0$, so that $\tilde F(0) = 0$ and $\D\tilde F(0) = \D F_\tau(\bu^*)$ because $\D\Phi(0) = \mathbb I$. Since $\Hess E(\bu^*)$ is self-adjoint with respect to $g_{\bu^*}$, the operator $\D\tilde F(0) = \mathbb I-\tau\Hess E(\bu^*)$ is self-adjoint with eigenvalues $1-\tau\rho_i$, and its operator norm in the norm $\norm\cdot_{\bu^*}$ on $T_{\bu^*}\mathcal M$ is $r_\tau$, which is smaller than one exactly when $0<\tau<2/\rho_{N-1}$. Let $0<\epsilon<1-r_\tau$, which suffices because the right side of \eqref{eqn:ostrowski-rate} increases with $\epsilon$. Since $\tilde F$ is smooth, there is a ball $B_\epsilon$ around $0$ in the norm $\norm\cdot_{\bu^*}$, which we take inside $\{\norm\xi_h\leq1\}$, with the operator norm $\norm{\D\tilde F(\xi)}_{\bu^*}\leq r_\tau+\epsilon$ for $\xi\in B_\epsilon$. The mean value inequality then gives $\norm{\tilde F(\xi)}_{\bu^*}\leq(r_\tau+\epsilon)\norm\xi_{\bu^*}$ on $B_\epsilon$, compare \citet[Prop.~3.1]{doi:10.1142/S0218202523500343} and \citet[Theorem~4.19]{boumal2023intromanifolds}. Let $\bu_0$ lie in the neighborhood $S_\epsilon := \Phi(B_\epsilon)$ of $\bu^*$, and let $\xi_0 := \Phi^{-1}(\bu_0)$ and $\xi_{k+1} := \tilde F(\xi_k)$. Then every $\xi_k$ lies in $B_\epsilon$ and $\norm{\xi_k}_{\bu^*}\leq(r_\tau+\epsilon)^k\norm{\xi_0}_{\bu^*}$. By the definition of $\tilde F$,
\[
\Phi(\xi_{k+1}) = F_\tau(\Phi(\xi_k)) = R_{\Phi(\xi_k)}\big(-\tau\grad E(\Phi(\xi_k))\big),
\]
and $\Phi(\xi_0) = \bu_0$, so the points $\Phi(\xi_k)$ are exactly the iterates $\bu_k$ of \eqref{eqn:iteration} with the constant step $\tau$. The fixed chart $\Phi$ only supplies coordinates for the iterates, and the retraction in $F_\tau$ stays based at the current iterate $\bu_k$. By Lemma \ref{lem:metric-equiv}, $\norm{\xi_k}_h\leq c_1^{-1/2}\norm{\xi_k}_{\bu^*}\leq c_1^{-1/2}(r_\tau+\epsilon)^k\norm{\xi_0}_{\bu^*}\leq\sqrt{c_2/c_1}\,(r_\tau+\epsilon)^k\norm{\xi_0}_h$. For $\bu\in\mathcal M$ and $\xi\in T_\bu\mathcal M$, the definition \eqref{eqn:iteration} and $\ip{\bu,\xi}_h = 0$ give
\begin{equation}\label{eqn:retraction-distance}
\norm{R_\bu(\xi)-\bu}_h^2 = 2\big(1-(1+\norm\xi_h^2)^{-1/2}\big) = \norm\xi_h^2\,\frac{2}{\sqrt{1+\norm\xi_h^2}\,\big(1+\sqrt{1+\norm\xi_h^2}\big)} ,
\end{equation}
where the last factor decreases from $1$ at $\xi = 0$ to $2-\sqrt2$ at $\norm\xi_h = 1$. Hence $\norm{R_\bu(\xi)-\bu}_h\leq\norm\xi_h$ for every $\xi\in T_\bu\mathcal M$, and $\norm{R_\bu(\xi)-\bu}_h\geq\norm\xi_h/2$ for $\norm\xi_h\leq1$. With $\bu = \bu^*$, the first bound gives $\norm{\bu_k-\bu^*}_h\leq\norm{\xi_k}_h$, and the second gives $\norm{\xi_0}_h\leq2\norm{\bu_0-\bu^*}_h$ because $\xi_0\in B_\epsilon$. Therefore
\[
\norm{\bu_k-\bu^*}_h\leq\norm{\xi_k}_h\leq\sqrt{c_2/c_1}\,(r_\tau+\epsilon)^k\norm{\xi_0}_h\leq2\sqrt{c_2/c_1}\,(r_\tau+\epsilon)^k\norm{\bu_0-\bu^*}_h ,
\]
which is \eqref{eqn:ostrowski-rate} with $C_\epsilon := 2\sqrt{c_2/c_1}$. Finally, $r_\tau$ is the maximum of the decreasing function $1-\tau\rho_1$ and the increasing function $\tau\rho_{N-1}-1$, so it is minimized where the two are equal, at $\tau^* = 2/(\rho_1+\rho_{N-1})$, and its minimum is $1-\tau^*\rho_1 = (\rho_{N-1}-\rho_1)/(\rho_{N-1}+\rho_1) = (\kappa(\bu^*)-1)/(\kappa(\bu^*)+1)$.
\end{proof}

For the $a_u$ metric, \eqref{eqn:rho-hweighted} reads
\[
\rho_{\bu^*}(\eta) = 1-\frac{\ip{(\lambda_1(A_\beta)\mathbb I-2\beta\diag((\bu^*)^2))\eta,\eta}_h}{\ip{A_\beta\eta,\eta}_h} .
\]
Up to the space discretization, the eigenvalues $1-\rho_i$ of $\D F_1(\bu^*)$ in \eqref{eqn:DF} are thus the eigenvalues of the bilinear form $(v,(\lambda-2\beta|u|^2)w)$ relative to the inner product $(v,w)_{a_u}$ on the $L^2$-orthogonal complement of $u$. This is the weighted eigenvalue problem in \citet{doi:10.1142/S0218202523500343}. The undamped inverse iteration $\bu_{k+1}\propto A_\beta(\bu_k)^{-1}\bu_k$ is the iteration \eqref{eqn:iteration} under the $a_u$ metric with $\tau_k = 1$, that is, $\bu_{k+1} = F_1(\bu_k)$. Under Assumption \ref{asp:monotone}, $\rho_{N-1}<2$ by Corollary \ref{coro:conditioning}, so Theorem \ref{thm:asymptotic-rate} applies with $\tau = 1$, and this iteration converges locally with the rate $r_1+\epsilon$ for every $\epsilon>0$, where $r_1 = \max_i|1-\rho_i|$ is the largest modulus of the eigenvalues of the weighted problem, in agreement with \citet[Remark~3.3]{doi:10.1142/S0218202523500343}. For a fixed step $\tau\in(0,2/\rho_{N-1})$ the eigenvalues of $\D F_\tau(\bu^*)$ are $1-\tau\rho_i = 1-\tau+\tau\mu_i$ with $\mu_i := 1-\rho_i$. Thus the rate $r_\tau = \max_i|1-\tau+\tau\mu_i|$ of Theorem \ref{thm:asymptotic-rate} is the sharper rate noted in \citet[Section~6.1]{doi:10.1142/S0218202523500343} for the damped inverse iteration with damping parameter $\tau$. By Corollary \ref{coro:conditioning}, $r_\tau\leq|1-\tau|+\tau\max_i|\mu_i|\leq|1-\tau|+\tau\lambda_1(A_\beta)/\lambda_2(A_\beta)$, which is the bound of \citet[Theorem~6.2]{doi:10.1142/S0218202523500343}. The critical damping parameter $2(1+\lambda_1(A_\beta)/\lambda_2(A_\beta))^{-1}$ of that bound is at most $2/\rho_{N-1}$, also by Corollary \ref{coro:conditioning}.

The same spectrum also controls the energy on a whole sublevel set. The proof combines sufficient decrease with a Polyak--{\L}ojasiewicz inequality, as in \citet[Lemma~11.28 and Theorem~11.29]{boumal2023intromanifolds} for geodesically strongly convex functions, with geodesic convexity replaced by bounds on the second derivative along retraction curves. \citet[Theorem~4.5.6]{absil2008optimization} prove an asymptotic energy contraction for Armijo steps along $-\grad E$, with a rate that also depends on the ratio of the extreme eigenvalues of the Hessian.

\begin{prop}[Energy contraction on a sublevel set]\label{thm:local-convergence}
Let Assumption \ref{asp:monotone} hold and let $\rho_- := \rho_1$, $\rho_+ := \rho_{N-1}$. There is $\epsilon>0$ such that, with $S_0 := \{\bu\in\mathcal M: E(\bu)\leq E(\bu^*)+\epsilon,\ \ip{\bu,\bu^*}_h>0\}$, the iteration
\[
\bu_{k+1} = R_{\bu_k}\Big(-\frac{1}{2\rho_+}\grad E(\bu_k)\Big)
\]
started from any $\bu_0\in S_0$ stays in $S_0$ and satisfies
\begin{equation}\label{eqn:kappa-rate}
E(\bu_k)-E(\bu^*)\ \leq\ \Big(1-\frac{1}{4\kappa(\bu^*)}\Big)^k\big(E(\bu_0)-E(\bu^*)\big).
\end{equation}
\end{prop}
\begin{proof}
\emph{Step 1.} The function $Q(t,\bu,\eta) := \frac{d^2}{dt^2}E(R_\bu(t\eta))$ is continuous on $[-1,1]\times K$, and $Q(0,\bu^*,\eta) = g_{\bu^*}(\Hess E(\bu^*)[\eta],\eta)\in[\rho_-,\rho_+]$ for $(\bu^*,\eta)\in K$ by \eqref{eqn:second-variation}. There are $r\in(0,1]$ and a neighborhood $\mathcal N$ of $\bu^*$ with $\rho_-/2\leq Q\leq2\rho_+$ for $|t|\leq r$, $\bu\in\mathcal N$, $(\bu,\eta)\in K$. Otherwise there would be $t_j\to0$ and $(\bu_j,\eta_j)\in K$ with $\bu_j\to\bu^*$ and $Q(t_j,\bu_j,\eta_j)\notin(\rho_-/2,2\rho_+)$. Since $K$ is compact, a subsequence of $(\bu_j,\eta_j)$ converges to some $(\bu^*,\eta^*)\in K$, and continuity of $Q$ gives $Q(0,\bu^*,\eta^*)\notin(\rho_-/2,2\rho_+)$, which contradicts $Q(0,\bu^*,\eta^*)\in[\rho_-,\rho_+]$. Now let $\bu\in\mathcal N$ and $\bw\in T_\bu\mathcal M$ with $0<\norm\bw_\bu\leq r$, and set $\eta := \bw/\norm\bw_\bu$. Since $\frac{d}{dt}E(R_\bu(t\eta))|_{t=0} = \D E(\bu)[\eta] = g_\bu(\grad E(\bu),\eta)$ by \eqref{eqn:iteration} and \eqref{eqn:projection}, Taylor's theorem for $t\mapsto E(R_\bu(t\eta))$ at $t = \norm\bw_\bu$ gives
\begin{equation}\label{eqn:taylor}
\tfrac{\rho_-}{4}\norm\bw_\bu^2\ \leq\ E(R_\bu(\bw))-E(\bu)-g_\bu(\grad E(\bu),\bw)\ \leq\ \rho_+\norm\bw_\bu^2 .
\end{equation}
\emph{Step 2.} Since $E(-\bu) = E(\bu)$ and $\pm\bu^*$ are the only global minimizers (Lemma \ref{lem:perron}), compactness of $\mathcal M$ shows that for small $\epsilon$ the sublevel set $\{E\leq E(\bu^*)+\epsilon\}$ lies in two disjoint balls of any prescribed radius around $\pm\bu^*$, and $S_0$ is its part in the ball around $\bu^*$. For $\bu\in S_0$ let $\xi_\bu := \bu^*/\ip{\bu,\bu^*}_h-\bu\in T_\bu\mathcal M$ be the tangent vector with $R_\bu(\xi_\bu) = \bu^*$ (Section \ref{sec:submanifold}). Since $\norm{\xi_\bu}_h^2 = \ip{\bu,\bu^*}_h^{-2}-1$, it satisfies $\norm{\xi_\bu}_\bu^2\leq c_2(\ip{\bu,\bu^*}_h^{-2}-1)$. We take $\epsilon$ so small that $S_0\subset\mathcal N$ and that $\norm{\grad E(\bu)}_\bu/(2\rho_+)<r$ and $\norm{\xi_\bu}_\bu<r$ on $S_0$, which is possible because both quantities are continuous and vanish at $\bu^*$.
\emph{Step 3.} Let $\bu\in S_0$ with $\mathbf g := \grad E(\bu)\neq0$. For $t\in(0,1]$ the vector $\bw_t := -t\mathbf g/(2\rho_+)$ satisfies $\norm{\bw_t}_\bu<r$ and $g_\bu(\mathbf g,\bw_t) = -t\norm{\mathbf g}_\bu^2/(2\rho_+)$, so the upper bound in \eqref{eqn:taylor} gives $E(R_\bu(\bw_t))\leq E(\bu)-t(2-t)\norm{\mathbf g}_\bu^2/(4\rho_+)\leq E(\bu)$. Hence the retraction segment from $\bu$ to $\bu_+ := R_\bu(\bw_1)$ stays in the sublevel set. Since the segment is connected and starts at $\bu\in S_0$, and the sublevel set lies in two disjoint balls, the segment stays in $S_0$, and $t = 1$ gives $E(\bu_+)\leq E(\bu)-\norm{\mathbf g}_\bu^2/(4\rho_+)$. The lower bound in \eqref{eqn:taylor} with $\bw = \xi_\bu$ gives
\[
E(\bu^*)\ \geq\ E(\bu)+g_\bu(\mathbf g,\xi_\bu)+\tfrac{\rho_-}{4}\norm{\xi_\bu}_\bu^2\ \geq\ E(\bu)-\norm{\mathbf g}_\bu^2/\rho_- ,
\]
because $g_\bu(\mathbf g,\xi)+\frac{\rho_-}{4}\norm\xi_\bu^2\geq-\norm{\mathbf g}_\bu^2/\rho_-$ for every $\xi$, with equality at $\xi = -2\mathbf g/\rho_-$. Hence $\norm{\mathbf g}_\bu^2\geq\rho_-(E(\bu)-E(\bu^*))$, and the two bounds give $E(\bu_+)-E(\bu^*)\leq(1-\rho_-/(4\rho_+))(E(\bu)-E(\bu^*))$. If $\grad E(\bu) = 0$ for some $\bu\in S_0$, then $\bu = \bu^*$, because otherwise the same lower bound gives $E(\bu^*)\geq E(\bu)+\frac{\rho_-}{4}\norm{\xi_\bu}_\bu^2>E(\bu)$. The iteration then stays at $\bu^*$, and \eqref{eqn:kappa-rate} follows by induction.
\end{proof}

\subsection{Global convergence}\label{sec:global}
\begin{thm}[Global convergence to a single critical point]\label{thm:global}
Let $(\bu_k)$ be generated by \eqref{eqn:iteration} from any $\bu_0\in\mathcal M$ with the constant step $\tau_k = 1/C_E$, and write $\mathbf g_k := \grad E(\bu_k)$. Then
\begin{equation}\label{eqn:suffdec}
E(\bu_{k+1})\ \leq\ E(\bu_k)-\frac{1}{2C_E}\norm{\mathbf g_k}_{\bu_k}^2 .
\end{equation}
In addition, $\grad E(\bu_k)\to0$ and the whole sequence converges to a single critical point $\bu_\infty\in\mathcal M$ with $\sum_k\norm{\bu_{k+1}-\bu_k}_h<\infty$. The same holds, with adjusted constants, for every step size rule satisfying \eqref{eqn:suffdec} and $\norm{\bu_{k+1}-\bu_k}_h\leq C\norm{\mathbf g_k}_{\bu_k}$, in particular for the exact line search $\tau_k = \arg\min_{\tau\in[0,\tau_{\max}]}E(R_{\bu_k}(-\tau\mathbf g_k))$ with $\tau_{\max}\geq1/C_E$.
\end{thm}
\begin{proof}
\emph{Step 1.} If $\mathbf g_k = 0$, then $\bu_{k+1} = \bu_k$ and \eqref{eqn:suffdec} is trivial. Otherwise, with $\eta := -\mathbf g_k/\norm{\mathbf g_k}_{\bu_k}$, $s_k := \norm{\mathbf g_k}_{\bu_k}/C_E\leq C_g$ and $\hat E(t) := E(R_{\bu_k}(t\eta))$, we have $\hat E'(0) = g_{\bu_k}(\mathbf g_k,\eta) = -\norm{\mathbf g_k}_{\bu_k}$ as in Step 1 of the proof of Proposition \ref{thm:local-convergence}, and $|\hat E''|\leq C_E$ on $[0,s_k]$ by the definition of $C_E$, because $(\bu_k,\eta)\in K$ and $s_k\leq C_g$. Taylor's theorem gives $E(\bu_{k+1})\leq E(\bu_k)-s_k\norm{\mathbf g_k}_{\bu_k}+\frac{C_E}2s_k^2$, which is \eqref{eqn:suffdec}, and the exact line search on $[0,\tau_{\max}]$ does at least as well as $\tau = 1/C_E$. By \eqref{eqn:retraction-distance}, $\norm{R_\bu(\xi)-\bu}_h\leq\norm\xi_h$ for $\xi\in T_\bu\mathcal M$. Since $\norm\xi_h\leq\norm\xi_\bu/\sqrt{c_1}$, the step $\xi = -\tau_k\mathbf g_k$ gives the step bound $\norm{\bu_{k+1}-\bu_k}_h\leq C\norm{\mathbf g_k}_{\bu_k}$ with $C = 1/(\sqrt{c_1}C_E)$ for $\tau_k = 1/C_E$ and $C = \tau_{\max}/\sqrt{c_1}$ for the exact line search. For a general rule, the step bound is assumed. Summing \eqref{eqn:suffdec} over $k$ and using $E\geq0$ gives $\sum_k\norm{\mathbf g_k}_{\bu_k}^2\leq2C_EE(\bu_0)<\infty$, so $\norm{\mathbf g_k}_h\leq c_1^{-1/2}\norm{\mathbf g_k}_{\bu_k}\to0$ by Lemma \ref{lem:metric-equiv}, and the energies decrease to a limit $E_\infty$. Since $\mathcal M$ is compact and $\grad E$ is continuous, the set $\mathcal A$ of accumulation points of $(\bu_k)$ is nonempty and compact, and every $\bv\in\mathcal A$ is a critical point with $E(\bv) = E_\infty$. For every metric except the Thomas--Fermi metric, the criticality of accumulation points and the convergence of the gradient to zero also follow from \citet[Theorem~4.3.1 and Corollary~4.3.2]{absil2008optimization} and \citet[Prop.~4.7]{boumal2023intromanifolds}, since both step rules achieve at least the decrease of the Armijo trial step $1/C_E$. For the modified $H^1$ metric, this sufficient-decrease argument parallels \citep[Section~5.1]{chen2024fully}.

\emph{Step 2.} Let $\bv\in\mathcal M$. The retraction $R_\bv$ is a real-analytic parametrization of $\mathcal M$ around $\bv$ by \eqref{eqn:iteration} and Section \ref{sec:submanifold}, and $E$ is a polynomial. Hence the {\L}ojasiewicz inequality \citep{loja1963} holds for the projected gradient \citep[Prop.~2.2]{schneider2015convergence}: there are $C_\bv>0$, $\varepsilon_\bv>0$ and $\theta_\bv\in(0,\frac12]$ with
\[
|E(\bu)-E(\bv)|^{1-\theta_\bv}\leq C_\bv\norm{\mathbf q(\bu)}_h \qquad\text{for } \bu\in\mathcal M,\ \norm{\bu-\bv}_h<\varepsilon_\bv ,
\]
where $\mathbf q(\bu) := A_\beta(\bu)\bu-\lambda(\bu)\bu$ is the gradient \eqref{eqn:projection} for the metric $B = \mathbb I$, that is, the projection of $A_\beta(\bu)\bu$ onto $T_\bu\mathcal M$ in $\ip{\cdot,\cdot}_h$. The change of variables $\bx = \mathbb M^{1/2}\bu$ identifies $\ip{\cdot,\cdot}_h$ with the Euclidean inner product, and \citet[Prop.~2.2]{schneider2015convergence} is applied in these coordinates, where the projected Euclidean gradient of $E$ is $\mathbb M^{1/2}\mathbf q(\bu)$ with norm $\norm{\mathbf q(\bu)}_h$. For $\zeta\in T_\bu\mathcal M$, Lemma \ref{lem:metric-equiv} gives $\ip{\mathbf q(\bu),\zeta}_h = \D E(\bu)[\zeta] = g_\bu(\grad E(\bu),\zeta)\leq\norm{\grad E(\bu)}_\bu\norm\zeta_\bu\leq\sqrt{c_2}\,\norm{\grad E(\bu)}_\bu\norm\zeta_h$, and $\zeta = \mathbf q(\bu)$ gives $\norm{\mathbf q(\bu)}_h\leq\sqrt{c_2}\,\norm{\grad E(\bu)}_\bu$. Hence
\begin{equation}\label{eqn:loja}
|E(\bu)-E(\bv)|^{1-\theta_\bv}\leq\sqrt{c_2}\,C_\bv\norm{\grad E(\bu)}_\bu \qquad\text{for } \bu\in\mathcal M,\ \norm{\bu-\bv}_h<\varepsilon_\bv .
\end{equation}
Now let $\bv_1,\ldots,\bv_m\in\mathcal A$ be such that the balls $\{\bu\in\mathcal M: \norm{\bu-\bv_i}_h<\varepsilon_{\bv_i}\}$ cover the compact set $\mathcal A$, let $\mathcal U$ be their union, $\theta := \min_i\theta_{\bv_i}$ and $C_L := \sqrt{c_2}\max_iC_{\bv_i}$. Since $E(\bv_i) = E_\infty$ and $t^{1-\theta}\leq t^{1-\theta_{\bv_i}}$ for $0\leq t\leq1$, \eqref{eqn:loja} gives
\begin{equation}\label{eqn:loja-uniform}
|E(\bu)-E_\infty|^{1-\theta}\leq C_L\norm{\grad E(\bu)}_\bu \qquad\text{for } \bu\in\mathcal U \text{ with } |E(\bu)-E_\infty|\leq1 .
\end{equation}
All but finitely many $\bu_k$ lie in $\mathcal U$, because otherwise a subsequence in the closed set $\mathcal M\setminus\mathcal U$ would have an accumulation point outside $\mathcal U$, which contradicts $\mathcal A\subset\mathcal U$.

\emph{Step 3.} If some $\mathbf g_{k_0} = 0$ the sequence is eventually constant. Otherwise $e_k := E(\bu_k)-E_\infty>0$ decreases to zero, and there is $k_0$ such that $\bu_k\in\mathcal U$ and $e_k\leq1$ for all $k\geq k_0$. For $k\geq k_0$, concavity of $r\mapsto r^\theta$ gives the first inequality, \eqref{eqn:suffdec} and \eqref{eqn:loja-uniform} give the second and the step bound gives the last inequality in
\begin{align*}
e_k^\theta-e_{k+1}^\theta\ &\geq\ \theta e_k^{\theta-1}(e_k-e_{k+1})\ \geq\ \frac{\theta}{C_L\norm{\mathbf g_k}_{\bu_k}}\cdot\frac{\norm{\mathbf g_k}_{\bu_k}^2}{2C_E}\\
&=\ \frac{\theta}{2C_LC_E}\norm{\mathbf g_k}_{\bu_k}\ \geq\ c_\star\norm{\bu_{k+1}-\bu_k}_h ,
\end{align*}
where $c_\star := \theta/(2C_LC_EC)$. Summing over $k\geq k_0$ gives $\sum_{k\geq k_0}\norm{\bu_{k+1}-\bu_k}_h\leq c_\star^{-1}e_{k_0}^\theta<\infty$. Hence $(\bu_k)$ is Cauchy and converges to a point $\bu_\infty$, which is then the only accumulation point, so $\bu_\infty\in\mathcal A$ is a critical point with $E(\bu_\infty) = E_\infty$.
\end{proof}

\section{Numerical experiments}\label{sec:numerics}

This section tests the predictions of Section \ref{sec:rayleigh} in three dimensions with up to $10^9$ unknowns, where the cost of applying each metric enters the comparison. Each metric is tested with Riemannian gradient descent, which Section \ref{sec:convergence} analyzes, and with the Riemannian conjugate gradient method. The convergence analysis does not cover conjugate gradients. However, the practical performance of conjugate gradient methods is related to the conditioning of the Hessian \citep{zheng2025threemetrics}, which Section \ref{sec:rayleigh} analyzes. The tests below show that both methods behave under the six metrics as the Hessian analysis suggests. All tests were run in Python with JAX on one NVIDIA GH200 GPU card, in double precision unless stated otherwise.

\subsection{Setup of the three-dimensional comparisons}\label{ssec:setup}

The domain is $\Omega = (-8,8)^3$ with $\beta = 1600$ and the stirrer potential
\begin{equation}\label{eqn:Vstir}
V(\bx) = x_1^2+x_2^2+4x_3^2+2w_0\,e^{-(x_1-r_0)^2-x_2^2},
\qquad w_0 = 30,\ r_0 = 1,
\end{equation}
a standard model of a laser stirrer in a condensate \citep{bao2004computing}. Its harmonic part is separable and is taken as $\mathbb V_1$, and the Gaussian beam, which couples $x_1$ and $x_2$, is $\mathbb V_2$. The discretization is the $Q^1$ element with Gauss--Lobatto quadrature on a uniform mesh, that is, the second-order finite difference scheme, with $n = 499$ interior points per direction, so $N = 499^3\approx1.2\times10^8$. Every method starts from the Thomas--Fermi profile $\bu_0\propto\sqrt{(\lambda_{\mathrm{TF}}-V)_+/\beta}$, with $\lambda_{\mathrm{TF}}$ defined by $\int_\Omega(\lambda_{\mathrm{TF}}-V)_+\,\ud\bx = \beta$. The shift in the modified $H^1$ metric is $\alpha = 20$. At $\beta = 1600$ the modified $H^1$ flow takes $582$, $207$, $169$ and $182$ iterations at $\alpha = 1$, $5$, $20$ and $100$, and its conjugate gradient variant takes $62$, $37$, $31$ and $32$. At $\alpha = 0$, which gives the classical $H^1$ seminorm, gradient descent does not reach $10^{-12}$ in $600$ iterations and conjugate gradients needs $114$. Among these values, $\alpha = 20$ gives the fewest iterations for both methods. Conjugate gradients is much less sensitive to $\alpha$ than gradient descent.

The six metrics of Table \ref{tab:metric-list} are compared. The Thomas--Fermi metric uses $\ell = 1/2$, and $\lambda(\bu)$ in \eqref{eqn:lambda-u} is recomputed from the energy evaluation without a solve. The operators $-\Delta_h+\alpha\mathbb I$ and $-\Delta_h+\mathbb V_1$ are separable for the potential \eqref{eqn:Vstir} and are applied exactly by tensor-product direct solvers \citep{liu2024simple,liu2026gpu}. The other four are applied by conjugate gradients preconditioned with $(-\Delta_h+\mathbb V_1)^{-1}$, stopped at relative residual $10^{-8}$.

Every method takes the step size that minimizes the energy along the retracted ray of its search direction $\eta_k$, $\tau_k = \arg\min_\tau E(R_{\bu_k}(\tau\eta_k))$, where $\eta_k = -\grad E(\bu_k)$ for gradient descent. This costs two applications of $-\Delta_h+\mathbb V$ and no solve, since the energy along the ray is a rational function of $\tau$ whose coefficients are inner products of $\bu_k$ and $\eta_k$. Conjugate gradients is the geometric conjugate gradient method \citep[Algorithm~13]{absil2008optimization} with the Fletcher--Reeves parameter \citep[(8.28)]{absil2008optimization}, in which each gradient norm is evaluated in the metric at its own iterate. The method transports the previous direction by the tangent projection \eqref{eqn:projection}, as in \citet[(8.10)]{absil2008optimization}, and restarts to steepest descent whenever the result is not a descent direction. Besides the wall clock time, we report the machine-independent number of applications of the tensor-product solver. An exact metric application counts as one, and an inexact one counts one per inner conjugate gradient iteration plus two for the separable warm start and the initial residual of each inner solve. All runs stop at $|E(\bu_k)-E^*|<10^{-12}$, where the reference $E^* = 9.012256553261389$ is computed on the same grid by a Davis--Yin splitting scheme \citep{li2026splitting}, which is not part of the comparison and is run until its residual stops decreasing. As a check of the discretization, Richardson extrapolation from the $499^3$ and $799^3$ grids reproduces the $Q^{20}$ value of Section \ref{ssec:results} to $6\times10^{-10}$. On the larger boxes of Sections \ref{ssec:tfbeta} and \ref{ssec:domain}, the step size of the Davis--Yin scheme decreases like $L^{-2}$, so the references for $\beta = 6400$ and $102400$ and for $L = 32$ are computed by Riemannian conjugate gradients in the modified $a_0$ metric, run until the residual stops decreasing.

\subsection{Comparison of the metrics}\label{ssec:results}

\begin{figure}[tp]
 \centering
 \includegraphics[width=\linewidth]{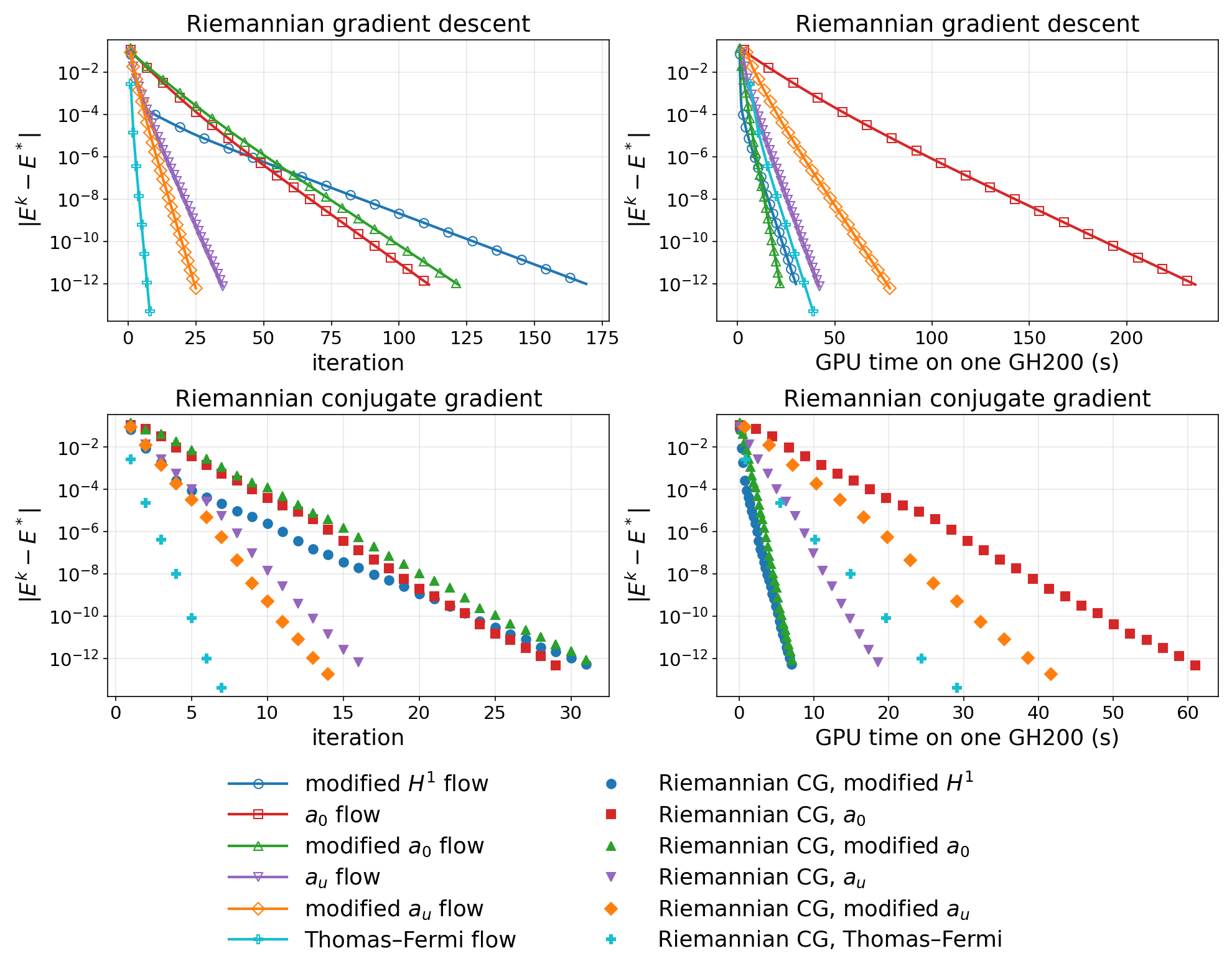}
 \caption{The metrics on the stirrer potential \eqref{eqn:Vstir} at $\beta = 1600$ from the Thomas--Fermi start, on a $499^3$ grid, with $\ell = 1/2$ in the Thomas--Fermi metric: the energy error against the iteration count on the left and against GPU time on the right, gradient descent in the top row and conjugate gradients in the bottom one. Color and marker indicate the metric, and the horizontal scales differ between the rows.}
 \label{fig:metrics}
\end{figure}

Figure \ref{fig:metrics} reports the comparison on the $499^3$ grid. Among the three classical metrics, the $a_u$ metric needs the fewest iterations but is not the fastest, because its operator depends on the iterate and is applied inexactly, with about $17$ preconditioner applications per iteration. The modified $a_0$ metric needs about as many iterations as the $a_0$ metric (Proposition \ref{prop:modified-a0}). Since each application of this metric is a single tensor-product solve, it needs $244$ applications of the solver against $3994$ for the $a_0$ metric and is an order of magnitude faster. Its conjugate gradient variant and that of the modified $H^1$ metric are the fastest methods. Conjugate gradients reduces the iteration count most where gradient descent is slowest, as expected if the rate of conjugate gradients is governed by $\sqrt\kappa$ rather than by $\kappa$.

The modified $a_u$ metric needs fewer iterations than the $a_u$ metric (Proposition \ref{prop:modified-au}). The halving of the condition number is asymptotic in $\beta$ and is not yet reached at $\beta = 1600$. Each iteration applies the metric twice, however, and $(-\Delta_h+\mathbb V_1)^{-1}$ preconditions $-\Delta_h+\mathbb V+3\beta\diag(\bu^2)$ less well than $A_\beta(\bu)$, so both variants are about twice as slow in wall clock time.

The Thomas--Fermi metric needs the fewest iterations. Although it needs two solves per iteration, it is faster than the $a_u$ flow, and among the gradient flows with an inexactly applied metric it uses the fewest solver applications. Its inner solves are the most expensive, $36.6$ preconditioner applications per metric application against $17.0$ for the $a_u$ metric and $23.7$ for the modified $a_u$ metric, because $-\Delta_h+\mathbb V_1$ is a poor approximation of an operator whose potential is clipped at $\lambda(\bu)$. At $\beta = 1600$ the small number of outer iterations compensates for this cost.

The iteration counts are essentially mesh independent. The five metrics other than the Thomas--Fermi metric were also run on a $99^3$ grid, with $125$ times fewer unknowns. On that grid, every iteration count differs from its value on the $499^3$ grid by at most one.

The exact spectrum on the stirrer \eqref{eqn:Vstir} at $\beta = 1600$, computed on the $499^3$ grid as in Appendix \ref{app:spectra}, accounts for the iteration counts of Figure \ref{fig:metrics}. The condition numbers are $1.30$ for the Thomas--Fermi metric, $4.73$ for the modified $a_u$ metric, $6.75$ for the $a_u$ metric, $20.8$ for the $a_0$ metric, $24.3$ for the modified $a_0$ metric and $39.9$ for the modified $H^1$ metric. If the rate $(\kappa(\bu^*)-1)/(\kappa(\bu^*)+1)$ of Theorem \ref{thm:asymptotic-rate} governs the number of iterations $k$ to a fixed tolerance, $2k\log((\kappa(\bu^*)+1)/(\kappa(\bu^*)-1))$ is the same for every metric. It lies between $20.1$ and $21.5$ for the four metrics with $4.7\leq\kappa(\bu^*)\leq24.3$, whose iteration counts differ by a factor of five, and between $25.5$ and $27.8$ for their conjugate gradient variants with $\sqrt{\kappa(\bu^*)}$ in place of $\kappa(\bu^*)$. The other two metrics deviate in opposite directions. The maximizing eigenvectors, computed on the $199^3$ grid, lie at the center of the trap for four metrics. For the modified $a_u$ metric the maximizing eigenvector is the grid-scale mode that attains the upper edge of \eqref{eqn:modified-au-spectrum} where $\mathbb V$ is largest, and for the modified $H^1$ metric it also lies where $\mathbb V$ is largest and the iterates are negligibly small, so the metric needs fewer iterations than its condition number predicts. The Thomas--Fermi metric needs more, since its few iterations are too few for the asymptotic regime to be reached.

To test the comparison on a second discretization, the same problem is solved with $Q^{20}$ spectral elements on $50$ cells per direction, so $n = 999$ and the problem has about $10^{9}$ unknowns. The mass matrix is then the diagonal of the Gauss--Lobatto weights rather than $h^3\mathbb I$, and the one-dimensional operators are diagonalized numerically. The reference $E^* = 9.012310296639$ is computed on the same grid by the Davis--Yin scheme, run until its residual stops decreasing, and the runs stop at $|E-E^*|<5\times10^{-12}$. Figure \ref{fig:metrics999} shows that the iteration counts carry over from the $499^3$ grid and are slightly lower throughout. The ordering by iterations is unchanged, and the two metrics with separable operators remain the cheapest per iteration. In wall clock time, however, the Thomas--Fermi flow is now slower than the $a_u$ flow, whereas at $499^3$ it was faster, and the difference comes from the inner solve.

\begin{figure}[tp]
 \centering
 \includegraphics[width=\linewidth]{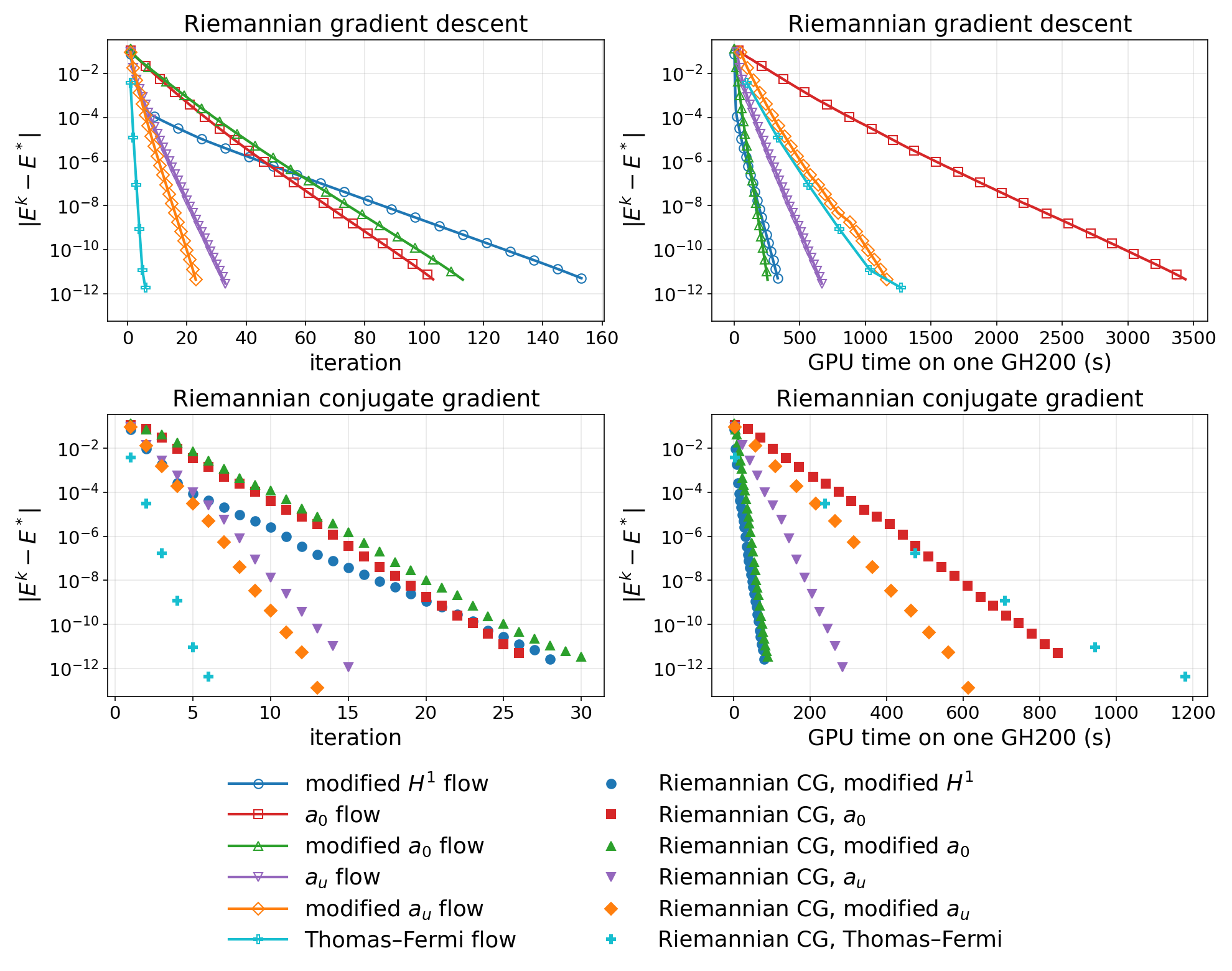}
 \caption{The same comparison with $Q^{20}$ spectral elements on a $999^3$ grid, about $10^{9}$ unknowns, in the panels and the drawing convention of Figure \ref{fig:metrics}.}
 \label{fig:metrics999}
\end{figure}

In double precision one vector of $999^3$ unknowns is $8$ gigabytes, so the metrics that invert a nonseparable operator fit into memory only if the Krylov vectors of the inner solve are held in single precision, while the residual is formed and the correction is accumulated in double precision, as in \citet{liu2026gpu}. At $n = 499$, where the double-precision solve also fits, the $a_0$ flow with the mixed-precision solve reaches the tolerance in the same number of iterations as with the double-precision one. Every method fits, with peak memory between $59.5$ and $78.1$ gigabytes of the $95.6$ available, once $B(\bu)\bx$ is consumed by a reduction and never stored. The Thomas--Fermi metric requires a further change of the inner solve. Its operator is the one that the separable preconditioner approximates worst, with about $75$ inner iterations against $29$ for $A_\beta(\bu)$. Over that many iterations, the recursively updated residual of the single-precision recurrence drifts so far from the true residual that the metric is applied far less accurately than the reported $10^{-6}$, and the outer iteration stalls with a collapsing step size. Capping the single-precision solve at $25$ iterations and recomputing the residual in double precision six times, at the same $10^{-6}$ target, restores convergence. The two Thomas--Fermi runs of Figure \ref{fig:metrics999} were computed with the capped solves, at about $2.4$ times the inner work per iteration.

\subsection{Dependence on the interaction strength}\label{ssec:tfbeta}

For the one-dimensional harmonic trap, Remark \ref{rem:tf-orders} shows that the condition number of the Thomas--Fermi metric is bounded in $\beta$, while that of the $a_u$ metric grows like $\lambda_1(A_\beta)^{2/3}$, those of its family \eqref{eqn:family} at least as fast, and the bound for the $a_0$ metric like $\lambda_1(A_\beta)^{5/3}$. We test whether the same holds for the stirrer potential in three dimensions. On the box $(-8,8)^3$ of Section \ref{ssec:results} the range of $\beta$ is limited, since the Thomas--Fermi radius grows like $\beta^{1/5}$ and leaves $(-8,8)^3$ near $\beta = 2\times10^4$. Thus we run the test on the larger box $(-16,16)^3$ with the same $499^3$ grid and the same initial value, and we vary only $\beta$, with $\beta = 1600$, $6400$ and $102400$, a factor of $64$. With $a\approx10^{-3}$ as in Section \ref{sec:gpe}, these values correspond to condensates of about $6.4\times10^4$ to $4.1\times10^6$ particles, inside the experimental range but well above the $\beta = 200$ used in \citet{bao2004computing}. Figures \ref{fig:betagd} and \ref{fig:betacg} show the error curves, one row per $\beta$.

Over that range $\lambda_1(A_\beta)$ grows by a factor of $4.68$, and the iteration counts follow the predicted growth of the condition numbers closely. The count of the $a_u$ flow grows by a factor of $2.74$, against the factor $\lambda_1^{2/3} = 2.80$ of Corollary \ref{thm:TF}, and that of its conjugate gradient variant grows by $1.69$, against the factor $\sqrt{2.80} = 1.67$ of a $\sqrt\kappa$ iteration count. The count of the modified $a_0$ flow grows faster than that of the $a_u$ flow and reaches the iteration cap at the largest $\beta$, in line with the growth $\lambda_1^{5/3} = 13.1$ of the bound of Proposition \ref{prop:modified-a0}. The count of the modified $a_u$ flow grows like that of the $a_u$ flow, with a smaller constant. The iteration count of the Thomas--Fermi metric does not grow. Its error curves have the same slope in the three rows, while every other curve flattens as $\beta$ grows, in agreement with \eqref{eqn:tf-uniform} and Remark \ref{rem:tf-orders}. At the largest $\beta$ the boundary layer is about $\lambda_1^{-1/6} = 0.46$ wide and is resolved by about seven cells of size $h = 0.064$. The resolution does not affect the comparison, since each run is measured against a reference on its own grid. The sweep thus measures the conditioning of the discrete problem in each metric, which is the quantity bounded in Section \ref{sec:rayleigh}.

The modified $H^1$ flow with the fixed shift $\alpha = 20$ behaves unlike the other two Euclidean metrics. Its iteration count changes little with $\beta$ and eventually falls below that of the $a_0$ flow. The three-dimensional sweep of Figure \ref{fig:spectrum3d} explains this through the negative part of $Z$. For the $a_0$ metric, $\zeta_-$ equals $\lambda_1(A_\beta)$ exactly in every run because the maximum in \eqref{eqn:defect} is attained where the density vanishes, whereas for the modified $H^1$ metric $\zeta_-$ is at most $\alpha$ to leading order in the Thomas--Fermi approximation $\beta\bu^2\approx(\lambda_1(A_\beta)-\mathbb V)_+$, because then $\mathbb V+3\beta\bu^2\geq\lambda_1(A_\beta)$ everywhere. The measured values are $19.4$, $18.2$, $16.8$ and $15.2$ against $\alpha = 20$ at $\beta = 10^2$, $10^3$, $10^4$ and $10^5$. Since $\sigma_T$ grows like $\lambda_1(A_\beta)^{1/3}$, the lower bound $\sigma_T/(\sigma_T+\zeta_-)$ of Theorem \ref{thm:rayleigh-quotient-bounds} tends to one for the modified $H^1$ metric and to zero for the $a_0$ metric, and the exact $\min\rho_\bu$ in Figure \ref{fig:spectrum3d} rises for the modified $H^1$ metric and falls for the $a_0$ metric. The upper edge, however, is governed by $\zeta_+$, which is of the order of $\max_i\mathbb V_i$. This suggests that the condition number of the modified $H^1$ metric is determined by the size of the box rather than by $\beta$ (Section \ref{ssec:domain}). The condition numbers of the modified $H^1$ and $a_0$ metrics cross between $\beta = 10^3$ and $10^4$.

\begin{figure}[tp]
 \centering
 \includegraphics[width=\linewidth,height=0.56\textheight,keepaspectratio]{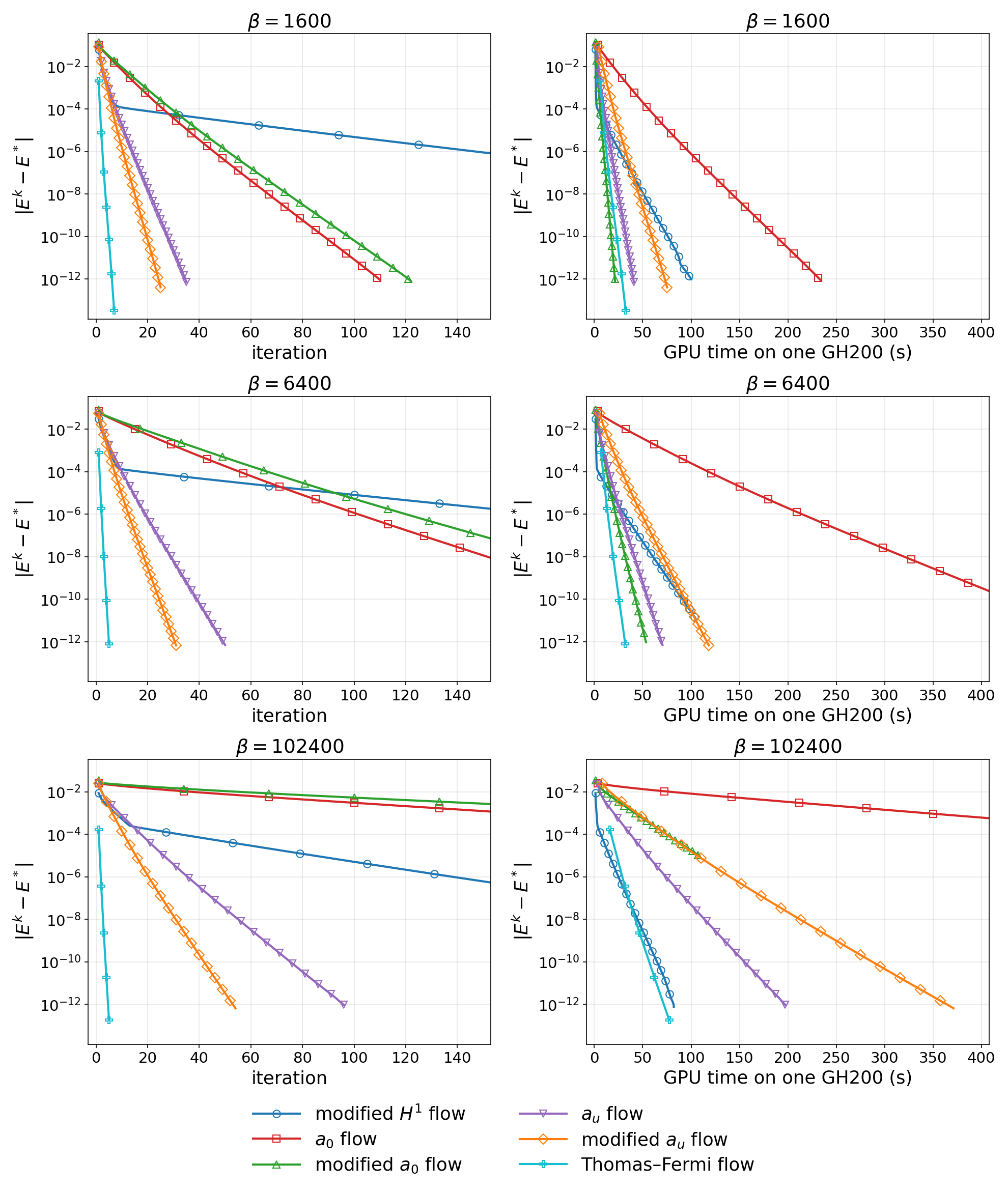}
 \caption{The six gradient flows against the interaction strength, on $(-16,16)^3$ at $499^3$, in the drawing convention of Figure \ref{fig:metrics}, with $\beta = 1600$, $6400$ and $102400$ from top to bottom and axes cut at $150$ iterations and $400$ seconds. A curve that leaves a panel does not converge inside it.}
 \label{fig:betagd}
\end{figure}

\begin{figure}[tp]
 \centering
 \includegraphics[width=\linewidth,height=0.56\textheight,keepaspectratio]{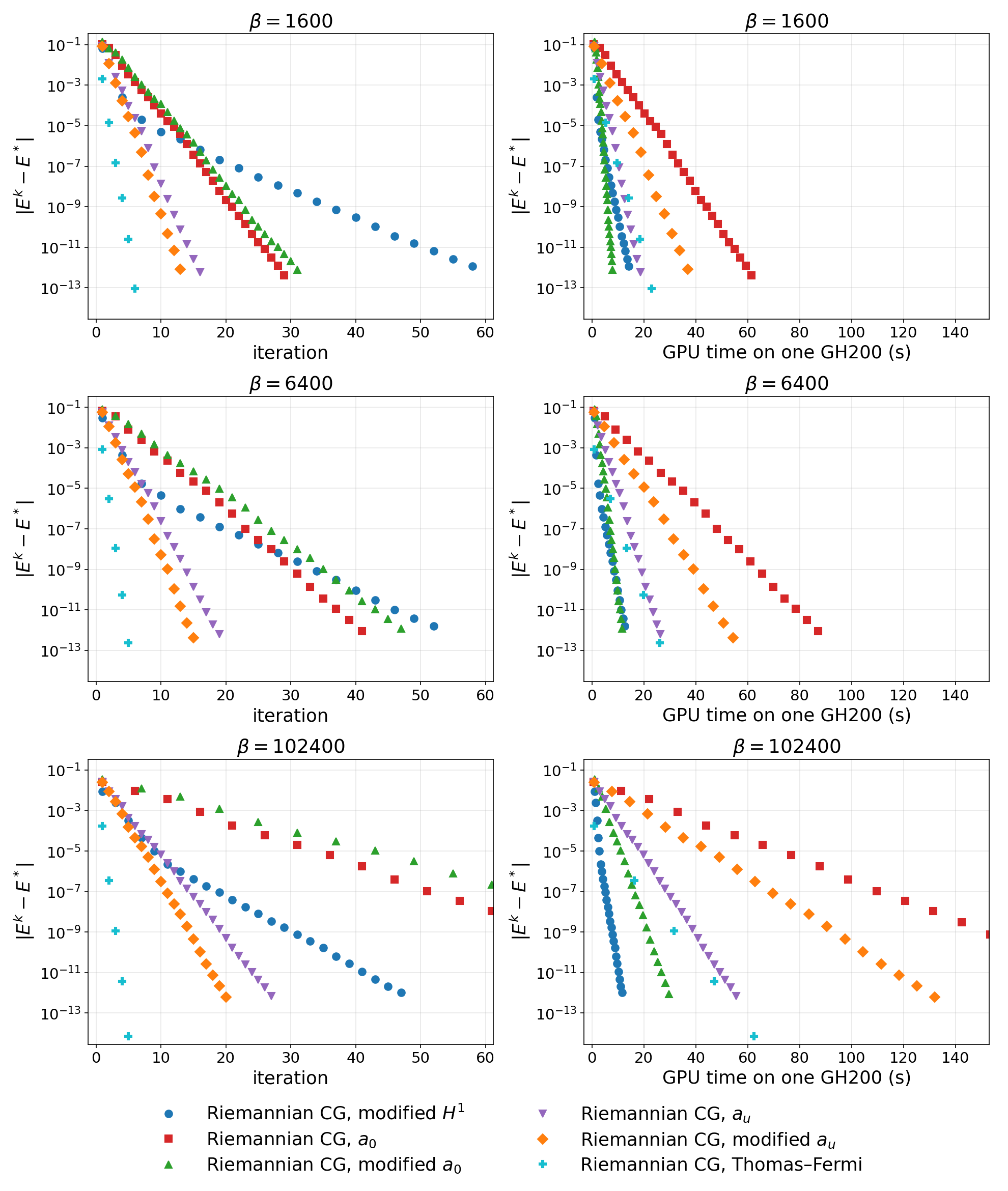}
 \caption{The same for the six conjugate gradient methods, cut at $60$ iterations and $150$ seconds.}
 \label{fig:betacg}
\end{figure}

\subsection{Dependence on the domain size}\label{ssec:domain}

Remark \ref{rem:domain-robust} bounds the condition number of the modified $a_0$ metric uniformly in the size of the box for a fixed mesh size, provided the coercivity constant stays bounded below, while the bound for the modified $H^1$ metric grows with the box. We test domain robustness numerically with a fixed number of unknowns instead of a fixed mesh size, on nested boxes $L = 8$, $16$ and $32$, with the stirrer potential \eqref{eqn:Vstir}, whose harmonic part is confining as the remark requires, and with the same $\beta = 1600$, initial value, $\alpha = 20$, inner solver and $499^3$ unknowns. The mesh is thus coarser on the larger boxes, with $h = 0.032$, $0.064$ and $0.128$. The run with $L = 16$ is the $\beta = 1600$ run of Section \ref{ssec:tfbeta}. Figures \ref{fig:domaingd} and \ref{fig:domaincg} show the error curves.

\begin{figure}[tp]
 \centering
 \includegraphics[width=\linewidth,height=0.52\textheight,keepaspectratio]{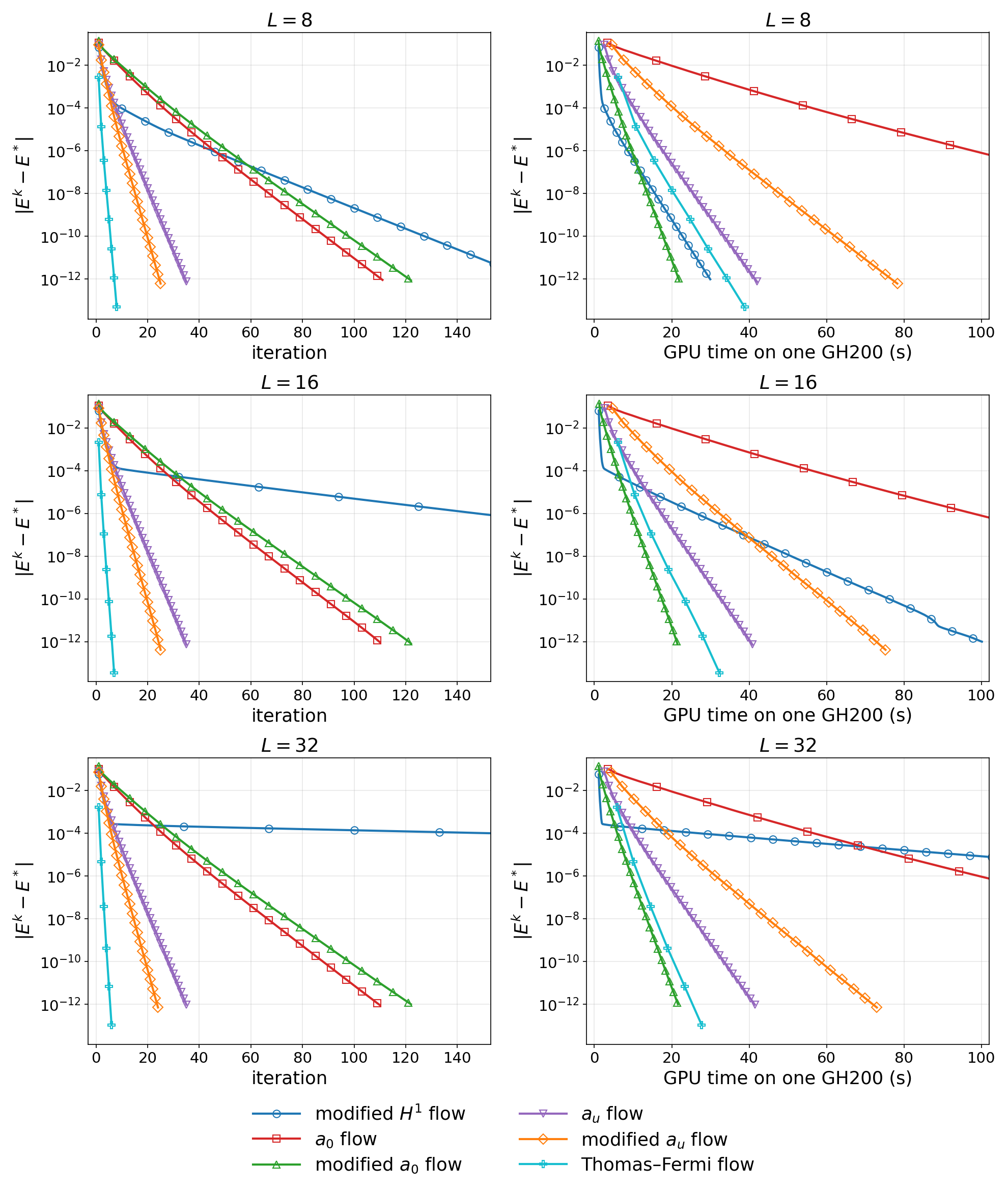}
 \caption{The six gradient flows on the nested boxes $(-L,L)^3$, each on a $499^3$ grid, in the drawing convention of Figure \ref{fig:metrics}: the energy error against the iteration count on the left and against GPU time on the right, with $L = 8$, $16$ and $32$ from top to bottom. All three rows use the same axes, cut at $150$ iterations and $100$ seconds, and a curve that leaves a panel does not converge inside it.}
 \label{fig:domaingd}
\end{figure}

\begin{figure}[tp]
 \centering
 \includegraphics[width=\linewidth,height=0.52\textheight,keepaspectratio]{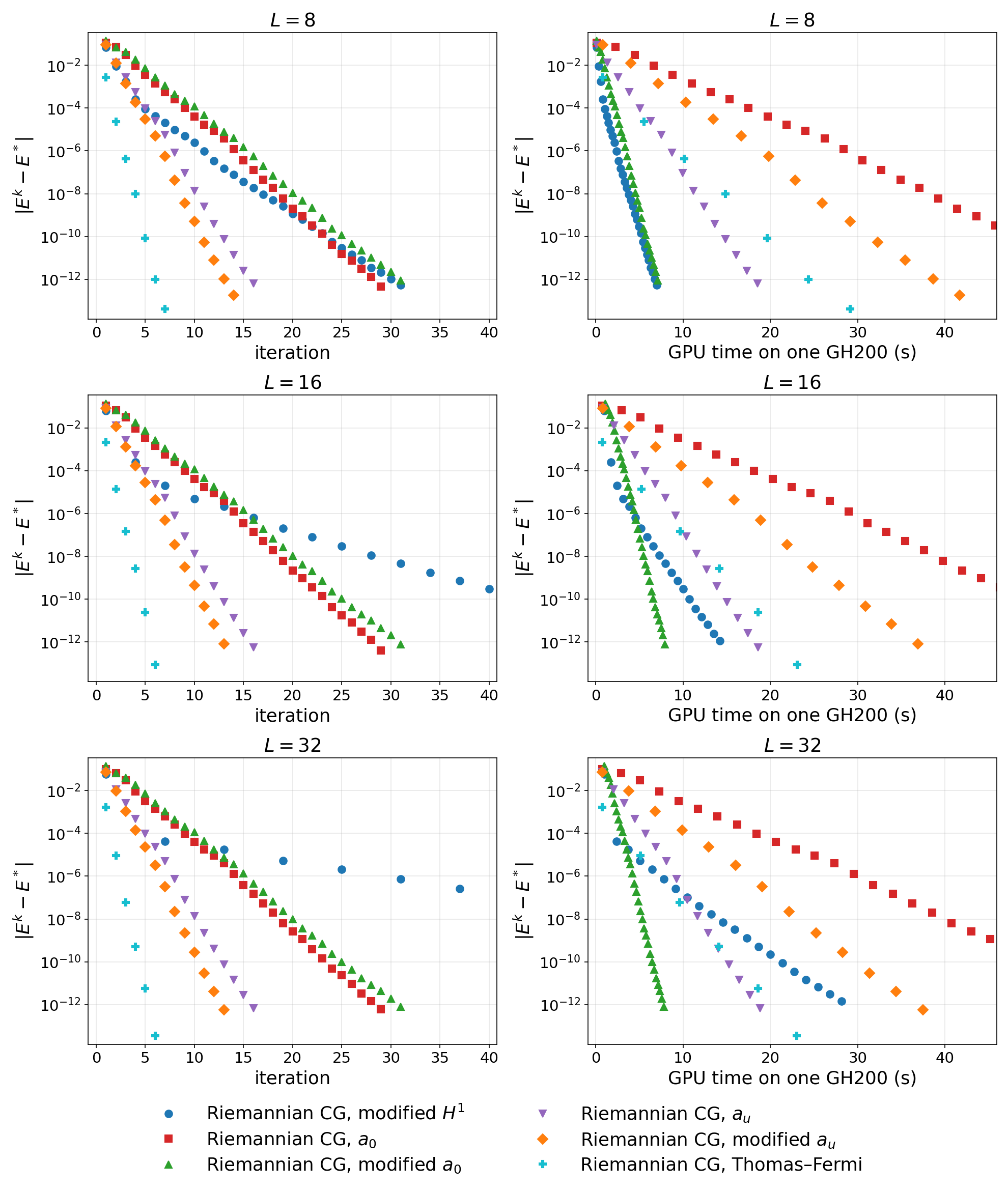}
 \caption{The same for the six conjugate gradient methods, on the same boxes and with the same convention, cut at $40$ iterations and $45$ seconds.}
 \label{fig:domaincg}
\end{figure}

For the five metrics built from the potential, the iteration counts are independent of $L$ to within one iteration for four of them and to within two iterations for the Thomas--Fermi flow. The error curves of these five metrics in Figures \ref{fig:domaingd} and \ref{fig:domaincg} are also nearly independent of $L$. The iteration count of the modified $H^1$ flow grows with $L$, and at $L = 32$ it does not reach the tolerance within the iteration cap. Its conjugate gradient variant needs about twice as many iterations per doubling of $L$. Since one iteration costs the same in all three boxes, every method except the two using the modified $H^1$ metric takes about the same wall clock time at $L = 32$ as at $L = 8$, or less. With conjugate gradients, the modified $a_0$ metric is as fast as the modified $H^1$ metric on the smallest box and the fastest of the six conjugate gradient methods on the two larger boxes. The $a_0$ metric is as robust in iterations as its modified version and an order of magnitude more expensive at every $L$.

\section{Conclusions}\label{sec:conclusions}
We have analyzed fully discretized Sobolev gradient flows for the Gross--Pitaevskii ground state problem through the condition number of the Riemannian Hessian at the ground state. The quadratic form of this Hessian does not depend on the metric, so bounds of its Rayleigh quotient give a uniform analysis of the flows for different metrics. The bound for the shifted Laplacian metric grows with the domain size, and so do its iteration counts in the tests. The condition number of the energy-adaptive metric grows like the two-thirds power of the ground state eigenvalue in the Thomas--Fermi regime, as proved for the one-dimensional harmonic trap and observed in two and three dimensions. Inspired by these bounds, we also consider three new metrics: a metric with only the separable part of the potential, which is robust in the domain size and inverted exactly by tensor-product solvers, a modified energy-adaptive metric that asymptotically halves the condition number, and a Thomas--Fermi metric whose condition number is bounded uniformly in the interaction strength, as proved for the one-dimensional harmonic trap and observed in three dimensions. For every metric, the local rate of the fixed-step iteration is determined by the spectrum of the Hessian, and the iterates converge to a single critical point from any initial value. The three-dimensional experiments are consistent with the analysis for both Riemannian gradient descent and Riemannian conjugate gradient methods. The Thomas--Fermi flow needs the fewest iterations. The shifted Laplacian metric and the metric with only the separable part of the potential are the fastest in computational time, because their iterations are the cheapest.

\section*{Acknowledgements}

This work used DeltaAI at the National Center for Supercomputing
Applications (NCSA) through allocation MTH260013 from the Advanced
Cyberinfrastructure Coordination Ecosystem: Services \& Support
(ACCESS) program, which is supported by U.S.\ National Science
Foundation grants \#2138259, \#2138286, \#2138307, \#2137603,
and \#2138296.

\section*{Funding}

Xiangxiong Zhang was supported in part by NSF DMS-2208518.

\section*{Use of AI-assisted technologies}

The authors have used generative AI tools such as ChatGPT and Claude
to assist mathematical discussions, algorithm implementations and drafting the manuscript. After
using AI tools, the authors have carefully reviewed, edited and
polished the manuscript, and take full responsibility for the
content.

\section*{Data Availability Statements}

The codes of all numerical tests are available on \url{https://github.com/zhan1966/GPE/}.

\bibliographystyle{abbrvnat}
\bibliography{reference}

\appendix
\setcounter{table}{0}
\section{Exact spectra of the Riemannian Hessian}\label{app:spectra}
We test the spectral predictions of Section \ref{sec:rayleigh} on $(-L,L)^d$ with $V=|\bx|^2$, using second-order finite differences and $\ell=1/2$ in the Thomas--Fermi metric. In one dimension, $N=3000$ for $\beta\leq100$ and $N=4000$ otherwise, with $L=\max\{6,1.25(3\beta/4)^{1/3}+4\}$. In two dimensions the grid is $200^2$ with $L=\max\{6,(2\beta/\pi)^{1/4}+3.5\}$, and in three dimensions it is $199^3$ with $L=1.25(15\beta/(8\pi))^{1/5}+4$. All three values of $L$ are derived from the Thomas--Fermi radius. At the largest $\beta$ in three dimensions, the layer of width $\lambda_1^{-1/6}$ has about $2.5$ grid points. Repeating the three-dimensional cases $\beta=10^5$ on $399^3$ and $\beta=3\times10^5$ on $299^3$ reproduces all condition numbers within $0.4\%$, except for the modified $a_u$ metric.

Ground states are computed by the $a_u$ flow with exact line search followed by Newton iterations until the residual reaches rounding error. Extreme Hessian eigenvalues are computed from the pencil $(\mathbb MH^*,\mathbb MB^*)$ on the tangent space, by a dense solver in one dimension, shift-invert Lanczos in two dimensions and locally optimal block preconditioned conjugate gradients with the preconditioner $A_1^{-1}$ in three dimensions. The code reproduces $\lambda_1/\lambda_2=0.99909$ and $\max_i|1-\rho_i|=0.94192$ from \citet[Section~7.2]{doi:10.1142/S0218202523500343} to five digits.

Table \ref{tab:spectrum} tests the bounds of Corollary \ref{coro:conditioning} for the $a_u$ metric in the harmonic trap $V = |\bx|^2$ in one, two and three dimensions. For $\beta\geq10^2$, the lower bound $1-\lambda_1(A_\beta)/\Lambda$ is within $9\%$ of the exact minimum in one dimension and within $21\%$ in two and three dimensions, the maximum approaches the upper bound $2$, and the bound on $\kappa_{a_u}(\bu^*)$ is within $2\%$ of the exact value for $\beta\geq10^4$ in one dimension. For comparison, the table also lists the bound $1-\lambda_1(A_\beta)/\lambda_2(A_\beta)$ that uses the spectral gap alone, which is known not to be sharp \citep[Remark~3.3]{doi:10.1142/S0218202523500343}. For $\beta\geq10^4$ this bound is smaller than the exact minimum by a factor between $31$ and $3\times10^4$, because in the Thomas--Fermi regime $\Lambda = \lambda_1(A_{3\beta})$, and $\lambda_1(A_{3\beta})-\lambda_1(A_\beta)$ grows with $\beta$ while the gap $\lambda_2(A_\beta)-\lambda_1(A_\beta)$ closes.

The table also confirms the asymptotics of Corollary \ref{thm:TF}. In one dimension, the slopes against $\lambda_1$ over $\beta\in[10^4,3\times10^5]$ are $-0.66$ for $\min\rho_{\bu^*}$, $0.333$ for $\lambda_1(A_{3\beta})-\lambda_1$ and $-1.00$ for the gap $\lambda_2(A_\beta)-\lambda_1$. The ratio $(\lambda_1(A_{3\beta})-\lambda_1)/\lambda_1^{1/3}$ is $2.40$, $2.41$, $2.41$ at $\beta=10^3,10^4,3\times10^5$, against $\mu_1 = 2.4105$ from collocation of the Painlev\'e-II equation and diagonalization of $M_0$. In two dimensions the ratio is $2.35$ at $\beta=10^4$.

\begin{table}[tp]
\centering\footnotesize
\setlength{\tabcolsep}{3pt}
\begin{tabular}{rrrrrrrrr}
\toprule
$\beta$ & $\lambda_1(A_\beta)$ & $\lambda_2-\lambda_1$ & $\lambda_1(A_{3\beta})-\lambda_1$ & $\min\rho_{\bu^*}$ & $1-\lambda_1/\lambda_2$ & $1-\lambda_1/\Lambda$ & $\max\rho_{\bu^*}$ & $\kappa_{a_u}(\bu^*)$ (bound) \\
\midrule
\multicolumn{9}{l}{$d = 1$, $N = 3000$ for $\beta \leq 10^2$ and $N = 4000$ otherwise} \\
$1$ & 1.383 & 1.82 & 0.734 & 0.6751 & 0.568 & 0.5680 & 1.0000 & 1.5 (2.5) \\
$10^2$ & 17.84 & 0.259 & 6.18 & 0.2806 & $1.4\times10^{-2}$ & 0.2573 & 1.7849 & 6.4 (7.7) \\
$10^3$ & 82.56 & 0.056 & 10.4 & 0.1167 & $6.8\times10^{-4}$ & 0.1123 & 1.9498 & 16.7 (17.8) \\
$10^4$ & 383.2 & 0.0121 & 17.5 & 0.04429 & $3.1\times10^{-5}$ & 0.04363 & 1.9890 & 44.9 (45.8) \\
$3\times10^5$ & 3699 & 0.00125 & 37.2 & 0.009998 & $3.4\times10^{-7}$ & 0.009964 & 1.9989 & 199.9 (200.7) \\
\midrule
\multicolumn{9}{l}{$d = 2$, $200^2$ grid} \\
$10$ & 3.241 & 1.49 & 1.85 & 0.5123 & 0.315 & 0.3637 & 1.0316 & 2.0 (4.6) \\
$10^2$ & 8.286 & 0.661 & 4.14 & 0.3839 & $7.4\times10^{-2}$ & 0.3336 & 1.4677 & 3.8 (5.8) \\
$10^3$ & 25.36 & 0.219 & 6.68 & 0.2251 & $8.6\times10^{-3}$ & 0.2084 & 1.8172 & 8.1 (9.6) \\
$10^4$ & 79.84 & 0.0695 & 10.1 & 0.1170 & $8.7\times10^{-4}$ & 0.1124 & 1.9456 & 16.6 (17.8) \\
\midrule
\multicolumn{9}{l}{$d = 3$, $199^3$ grid} \\
$10^2$ & 5.932 & 1.10 & 2.77 & 0.4001 & 0.157 & 0.3187 & 1.2128 & 3.0 (5.8) \\
$10^3$ & 13.30 & 0.493 & 4.76 & 0.2972 & $3.6\times10^{-2}$ & 0.2635 & 1.6090 & 5.4 (7.5) \\
$10^4$ & 32.58 & 0.200 & 7.07 & 0.1911 & $6.1\times10^{-3}$ & 0.1782 & 1.8431 & 9.6 (11.2) \\
$10^5$ & 81.44 & 0.0797 & 10.0 & 0.1137 & $9.8\times10^{-4}$ & 0.1095 & 1.9403 & 17.1 (18.3) \\
\bottomrule
\end{tabular}
\caption{Exact extreme eigenvalues $\min\rho_{\bu^*}$ and $\max\rho_{\bu^*}$ of the Riemannian Hessian at the ground state under the $a_u$ metric for the harmonic trap $V = |\bx|^2$ in $d = 1,2,3$ dimensions, the lower bound $1-\lambda_1(A_\beta)/\lambda_2(A_\beta)$ based on the spectral gap alone, the lower bound $1-\lambda_1(A_\beta)/\Lambda$ of Corollary \ref{coro:conditioning}, and the exact condition number $\kappa_{a_u}(\bu^*)$ with the bound of Corollary \ref{coro:conditioning} in parentheses.}
\label{tab:spectrum}
\end{table}

Table \ref{tab:tf-spectrum} compares the modified $a_u$ and Thomas--Fermi metrics with the $a_u$ metric in one dimension, for $V(x)=x^2$ and for $V(x)=x^2+60e^{-(x-1)^2}$, the one-dimensional analogue of the stirrer potential \eqref{eqn:Vstir} of Section \ref{sec:numerics}. At $\beta=10^2,10^3,10^4$ in the harmonic trap, gradient descent with exact line search on a grid with $N=2000$, started from the Thomas--Fermi profile and stopped at $E(\bu_k)-E(\bu^*)<10^{-12}$, takes $21$, $47$, $114$ iterations under the modified $a_u$ metric against $32$, $84$, $217$ under the $a_u$ metric, but it costs $42$, $94$, $228$ solves against $32$, $84$, $217$. The halved condition number thus reduces iterations without reducing solves.

For the Thomas--Fermi metric, the harmonic-trap data verify the two hypotheses of \eqref{eqn:tf-uniform}, with $e_{\mathrm{TF}}/\lambda_1^{1/3}\to0.43$ from below and $\sigma_T/\lambda_1^{1/3}\to\mu_1$. At $\beta=10^4$, \eqref{eqn:tf-kappa} gives $\kappa_{\mathrm{TF}}(\bu^*)\leq1.38$ against the measured $1.16$, a factor $39$ smaller than $\kappa_{a_u}(\bu^*)$. With the nonseparable bump, $\kappa_{\mathrm{TF}}(\bu^*)$ stays between $1.16$ and $1.18$. Figure \ref{fig:spectrum3d} extends the comparison to three dimensions over $\beta\in[1,3\times10^5]$. There $\kappa_{\mathrm{TF}}(\bu^*)$ falls from $1.89$ to $1.17$, while $\kappa_{a_0}(\bu^*)$ grows by a factor $330$, and the measured exponents against $\lambda_1$ are $0.62$, $1.62$ and $-0.12$ for the $a_u$, $a_0$ and Thomas--Fermi metrics, consistent with the predicted $2/3$, with the $5/3$ in the $a_0$ bound and with \eqref{eqn:tf-uniform}.

The upper edge $s_{N-1}/(s_{N-1}+\lambda_1)$ of the modified $a_u$ spectrum is mesh dependent by construction, and refinement changes $\kappa_{\widetilde a_u}(\bu^*)$ at $\beta=10^4$ from $21.5$ with $N=2000$ to $22.5$ with $N=4000$ in one dimension and by at most $4\%$ in the three-dimensional mesh checks. The lower edge is stable under refinement.

\begin{table}[tp]
\centering\footnotesize
\setlength{\tabcolsep}{3pt}
\begin{tabular}{lrrrrrrr}
\toprule
$V(x)$ & $\beta$ & $\lambda_1$ & $e_{\mathrm{TF}}/\lambda_1^{1/3}$ & $\sigma_T/\lambda_1^{1/3}$ & $a_u$ & $\widetilde a_u$ & TF \\
\midrule
$x^2$ & $10^2$ & $17.84$ & $0.39$ & $2.37$ & $6.4$ & $3.9$ & $1.17$\\
$x^2$ & $10^3$ & $82.56$ & $0.41$ & $2.40$ & $16.7$ & $8.9$ & $1.16$\\
$x^2$ & $10^4$ & $383.16$ & $0.42$ & $2.41$ & $44.9$ & $22.5$ & $1.16$\\
\midrule
$x^2 + 60e^{-(x-1)^2}$ & $10^2$ & $24.93$ & $0.95$ & $2.38$ & $6.7$ & $4.6$ & $1.18$\\
$x^2 + 60e^{-(x-1)^2}$ & $1600$ & $117.88$ & $0.41$ & $2.40$ & $20.4$ & $11.0$ & $1.16$\\
$x^2 + 60e^{-(x-1)^2}$ & $10^4$ & $385.87$ & $0.42$ & $2.41$ & $44.7$ & $22.6$ & $1.16$\\
\bottomrule
\end{tabular}
\caption{Exact condition numbers $\kappa(\bu^*)$ of the Riemannian Hessian at the ground state in one dimension (last three columns) under the $a_u$, modified $a_u$ and Thomas--Fermi ($\ell = 1/2$) metrics, with the Thomas--Fermi error \eqref{eqn:tf-defect} and the coercivity constant $\sigma_T$ in units of the layer scale $\lambda_1^{1/3}$, for the harmonic trap and for the trap with the stirrer bump.}
\label{tab:tf-spectrum}
\end{table}

\begin{figure}[tp]
 \centering
 \includegraphics[width=\linewidth]{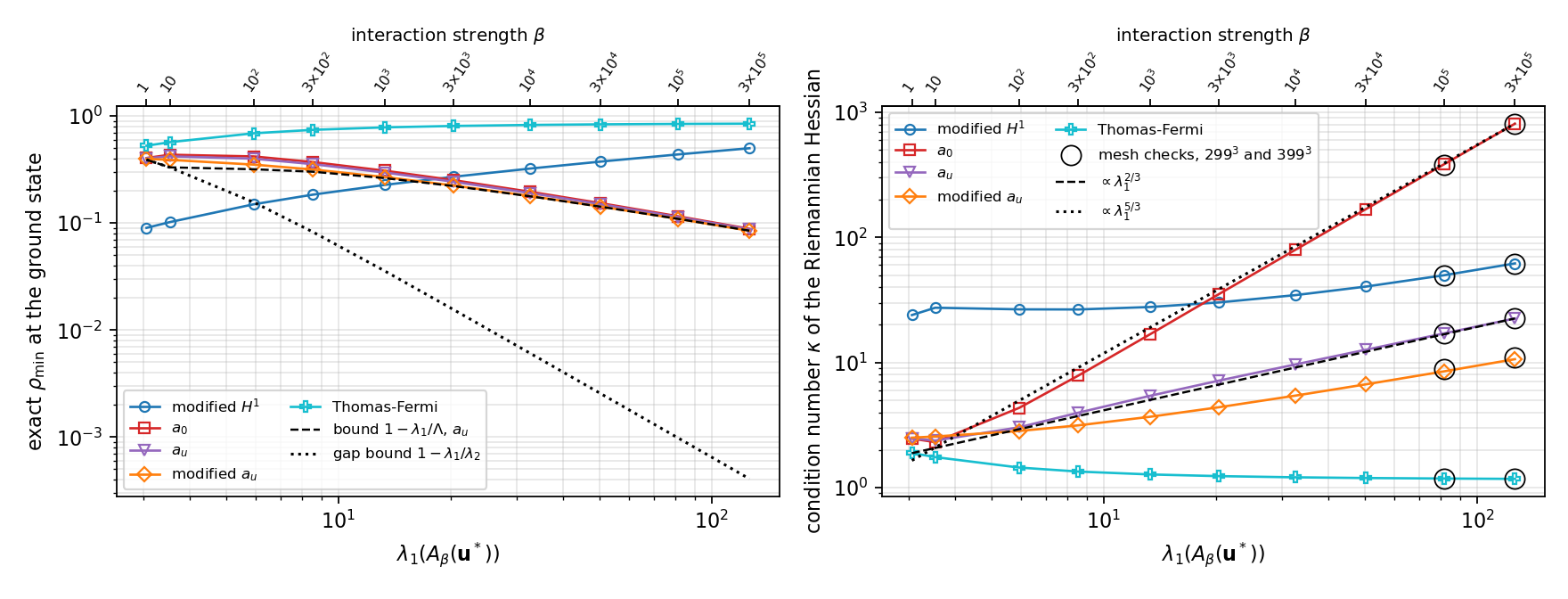}
 \caption{Three-dimensional harmonic trap on a $199^3$ grid. The open rings are the two mesh checks, $\beta = 10^5$ on $399^3$ and $\beta = 3\times10^5$ on $299^3$. Both panels are drawn against $\lambda_1(A_\beta(\bu^*))$, with $\beta$ on the top axis. Left: the exact $\min\rho_{\bu^*}$ for the six metrics of Table \ref{tab:metric-list}, with $\ell = 1/2$ in the Thomas--Fermi metric, together with the lower bound $1-\lambda_1(A_\beta)/\Lambda$ of Corollary \ref{coro:conditioning} and the bound based on the spectral gap alone, both for the $a_u$ metric. The modified $a_0$ metric coincides with the $a_0$ metric because $\mathbb V$ is separable here. Right: the exact condition numbers against the reference slopes $\lambda_1(A_\beta)^{2/3}$ and $\lambda_1(A_\beta)^{5/3}$.}
 \label{fig:spectrum3d}
\end{figure}

\section{Riemannian connection, Hessian and vector transport}\label{app:riemannian-connection}
We work on an open set where $B(\bu)$ is smooth and $\mathbb M B(\bu)$ is symmetric positive definite. Write $\mathbb G(\bu):=B(\bu)^{-1}$ and let $\odot$ denote componentwise multiplication. For $\bu\in\mathcal M$, the $g_\bu$-orthogonal projection $P_\bu$ is defined in \eqref{eqn:projection}, and by \eqref{eqn:metric} and \eqref{eqn:tangent-normal-space}
\begin{equation}\label{eqn:MG-symmetry}
g_\bu(\mathbb G(\bu)\bx,\by)=\ip{\bx,\by}_h,
\qquad N_\bu\mathcal M=\operatorname{span}\{\mathbb G(\bu)\bu\}.
\end{equation}
The five metrics of \eqref{eqn:six-metrics} other than the Thomas--Fermi metric have the form
\begin{equation}\label{eqn:appendix-metric-family}
B(\bu)=B_0+c\beta\diag(\bu^2),\qquad B_0\text{ independent of }\bu.
\end{equation}
Here $c=0$ for the three Euclidean metrics, $c=1$ for the $a_u$ metric and $c=3$ for the modified $a_u$ metric, and $B_0=-\Delta_h+\mathbb V$ gives the family \eqref{eqn:family}. The coefficient $c$ enters only the metric and not the energy \eqref{eqn:discrete-energy}.

\begin{lem}[Riemannian connection]\label{lem:riemannian-connection-ambient}
For a smooth metric $B$, the ambient Levi--Civita connection is $\overline\nabla_\xi\chi=\D\chi(\bu)[\xi]+\Gamma_\bu(\xi,\chi)$, where the bilinear map $\Gamma_\bu$ is determined by
\begin{equation}\label{eqn:connection-general}
\begin{aligned}
2\ip{B\Gamma_\bu(\xi,\chi),\zeta}_h
={}&\ip{\D B[\xi]\chi,\zeta}_h+\ip{\D B[\chi]\xi,\zeta}_h\\
&-\ip{\D B[\zeta]\xi,\chi}_h
\end{aligned}
\end{equation}
for all $\zeta\in\mathbb R^N$, with all vector-field values and coefficients evaluated at $\bu$. For \eqref{eqn:appendix-metric-family},
\begin{equation}\label{eqn:connection-euclidean}
\overline\nabla_\xi\chi
=\D\chi(\bu)[\xi]+c\beta\mathbb G(\bu)(\bu\odot\xi\odot\chi).
\end{equation}
The induced connection on $\mathcal M$ is $\nabla_\xi\chi=P_\bu(\overline\nabla_\xi\chi)$ for tangent vector fields.
\end{lem}
\begin{proof}
The Koszul formula \citep[(5.11)]{absil2008optimization} for constant vector fields gives \eqref{eqn:connection-general}, and the term $\D\chi(\bu)[\xi]$ accounts for general vector fields \citep[(5.7)]{absil2008optimization}. For \eqref{eqn:appendix-metric-family}, $\D B[\xi]=2c\beta\diag(\bu\odot\xi)$. Since $\mathbb M$ is diagonal, each inner product on the right of \eqref{eqn:connection-general} equals $2c\beta\sum_iw_i u_i\xi_i\chi_i\zeta_i$. Their signed sum gives \eqref{eqn:connection-euclidean}. Orthogonal projection gives the induced connection \citep[Prop.~5.3.2]{absil2008optimization}.
\end{proof}

\begin{lem}[Riemannian Hessian at an arbitrary point]\label{lem:riemannian-hess}
Let $\bu\in\mathcal M$, $\eta\in T_\bu\mathcal M$, and $p_\bu:=\grad E(\bu)=\mathbb G(A_\beta(\bu)\bu-\lambda_B(\bu)\bu)$, with $\lambda_B$ as in \eqref{eqn:projection}. For a smooth metric,
\begin{equation}\label{eqn:riemannian-hessian-general}
\Hess E(\bu)[\eta]
=P_\bu\Big(\mathbb G(A_{3\beta}(\bu)-\lambda_B(\bu)\mathbb I)\eta
 -\mathbb G\D B[\eta]p_\bu+\Gamma_\bu(\eta,p_\bu)\Big).
\end{equation}
For \eqref{eqn:appendix-metric-family},
\begin{equation}\label{eqn:riemannian-hess-flat}
\Hess E(\bu)[\eta]
=P_\bu\mathbb G\Big((A_{3\beta}(\bu)-\lambda_B(\bu)\mathbb I)\eta
 -c\beta\,\bu\odot\eta\odot p_\bu\Big).
\end{equation}
At a critical point, $p_\bu=0$ and $\lambda_B(\bu)=\lambda(\bu)$, so for every smooth metric
\[
g_\bu(\Hess E(\bu)[\eta],\eta)
=\ip{(A_{3\beta}(\bu)-\lambda(\bu)\mathbb I)\eta,\eta}_h,
\]
which recovers Lemma \ref{lem:second-variation}.
\end{lem}
\begin{proof}
Differentiating $p_\bu$ along a curve on $\mathcal M$ with initial velocity $\eta$ gives
\[
\D p_\bu[\eta]
=\mathbb G(A_{3\beta}(\bu)-\lambda_B(\bu)\mathbb I)\eta
 -\mathbb G\D B[\eta]p_\bu
 -\D\lambda_B(\bu)[\eta]\,\mathbb G\bu.
\]
The last term is normal and vanishes under $P_\bu$. By \citet[Def.~5.5.1 and Prop.~5.3.2]{absil2008optimization}, $\Hess E(\bu)[\eta]$ is the projection $P_\bu$ of the ambient covariant derivative of the vector field $p$ along $\eta$, so adding the connection term and projecting proves \eqref{eqn:riemannian-hessian-general}, and \eqref{eqn:connection-euclidean} gives \eqref{eqn:riemannian-hess-flat}. At a critical point the metric-derivative terms vanish, and \eqref{eqn:MG-symmetry} gives the quadratic-form identity. Away from critical points, $\lambda_B(\bu)$ cannot in general be replaced by $\lambda(\bu)$.
\end{proof}

\emph{The Thomas--Fermi metric.}
The operator $B_{\mathrm{TF}}(\bu)$ of \eqref{eqn:tf-metric} depends on $\bu$ also through $\lambda(\bu)$, so it is not of the form \eqref{eqn:appendix-metric-family}. It is smooth on the open set of all $\bu\neq0$ with $\lambda(\bu)\neq\mathbb V_i$ for every node $i$. On this set, Lemmas \ref{lem:riemannian-connection-ambient} and \ref{lem:riemannian-hess} apply with the general formulas \eqref{eqn:connection-general} and \eqref{eqn:riemannian-hessian-general}.

\emph{Retraction and vector transport.}
The projection transport used in Section \ref{ssec:setup} is
\begin{equation}\label{eqn:projection-transport}
\mathcal T_\eta(\xi):=P_\bv\xi
=\xi-\frac{\ip{\bv,\xi}_h}{\ip{\bv,\mathbb G(\bv)\bv}_h}\,\mathbb G(\bv)\bv,
\qquad\bv=R_\bu(\eta),
\end{equation}
for $\eta,\xi\in T_\bu\mathcal M$. It is the transport \citep[(8.10)]{absil2008optimization} with the $g$-orthogonal projection in place of the Euclidean one. It lies in $T_\bv\mathcal M$, is linear in $\xi$, satisfies $\mathcal T_0(\xi)=\xi$ and depends smoothly on $(\bu,\eta,\xi)$ wherever $B$ is smooth. These are the defining properties of a vector transport associated with $R$ \citep[Def.~8.1.1]{absil2008optimization}. Differentiating the retraction gives another transport \citep[(8.6) and Example~8.1.4]{absil2008optimization},
\begin{equation}\label{eqn:differentiated-transport}
\D R_\bu(\eta)[\xi]
=\frac{\xi-\bv\ip{\bv,\xi}_h}{\sqrt{1+\norm\eta_h^2}},
\end{equation}
which differs from \eqref{eqn:projection-transport} in general, and neither one is in general the parallel transport of the Riemannian connection. In the metric norm $\norm\xi_\bu:=g_\bu(\xi,\xi)^{1/2}$ of Section \ref{sec:convergence}, orthogonal projection gives $\norm{\mathcal T_\eta(\xi)}_\bv\leq\norm\xi_\bv$, but in general not $\norm{\mathcal T_\eta(\xi)}_\bv\leq\norm\xi_\bu$ when the metric depends on $\bu$. Lemma \ref{lem:second-variation} holds along every curve, in particular along $t\mapsto R_\bu(t\eta)$, because the acceleration term is multiplied by $\grad E(\bu)=0$ at a critical point, so no second-order property of $R$ is needed.

\end{document}